\documentclass[11pt]{article}
\usepackage{hyperref}
\usepackage{amsfonts,latexsym,rawfonts,amsmath,amssymb,amsthm}
\usepackage{amsmath,amssymb,amsfonts,latexsym,lscape,rawfonts}
\usepackage[cm]{fullpage}
\usepackage{amsmath,amscd, float,times,rotating}
\usepackage{pb-diagram}
\usepackage{color}
\numberwithin{equation}{section}
\newcommand{\pd}[2]{\frac {\partial #1}{\partial #2}}

\newcommand{\al}{\alpha}
\newcommand{\bb}{\beta}
\newcommand{\la}{\lambda}
\newcommand{\La}{\Lambda}
\newcommand{\oo}{\omega}

\newcommand{\dd}{\delta}
\newcommand{\Na}{\nabla}

\def\ga{\gamma}

\newcommand{\ee}{\epsilon}

\newcommand{\Si}{\Sigma}

\newcommand{\beq}{\begin{equation}}
\newcommand{\eeq}{\end{equation}}
\newcommand{\beqs}{\begin{eqnarray*}}
\newcommand{\eeqs}{\end{eqnarray*}}
\newcommand{\beqn}{\begin{eqnarray}}
\newcommand{\eeqn}{\end{eqnarray}}
\newcommand{\beqa}{\begin{array}}
\newcommand{\eeqa}{\end{array}}

\def\p{\partial}

\def\RR{{\mathbb R}}

\def\PP{{\mathbb P}}
\def\CC{{\mathbb C}}

\def\TT{{\mathbb T}}
\def\ri{\rightarrow}

\def\un{\underline}

\def\vol{{\rm Vol}}

\def\Re{{\rm Re}}
\def\Im{{\rm Im}}

\def\cA{{\mathcal A}}

\def\cB{{\mathcal B}}

\def\n{{\nonumber}}
\newtheorem{prop}{Proposition}[section]
\newtheorem{theo}[prop]{Theorem}
\newtheorem{lem}[prop]{Lemma}

\newtheorem{rem}[prop]{Remark}

\newtheorem{defi}[prop]{Definition}

\title{Convergence of Lagrangian mean curvature flow in K\"ahler-Einstein
manifolds}
\author{Haozhao Li\footnote{Research supported in part by National Science Foundation
of China No. 11001080 and a startup funding from University of
Science and Technology of China.
 }}

\begin{document}
\bibliographystyle{plain}
\date{}
\maketitle

\tableofcontents


\section{Introduction}
    The problem on Lagrangian submanifolds in Calabi-Yau manifolds or general K\"ahler manifolds
has been the subject of intense study over the last few decades. They are important both
in mathematics and physics because
minimal Lagrangian submanifolds in Calabi-Yau manifolds are related to T-duality and Mirror symmetry in physics
in the fundamental paper \cite{[SYZ]}. However, to construct a minimal Lagrangian submanifold is very
difficult. Here we will use Lagrangian mean curvature flow to give some sufficient conditions for the existence
of minimal Lagrangian submanifolds in a general K\"ahler-Einstein manifold.

\medskip
A mean curvature flow is called Lagrangian mean curvature flow if the initial submanifold is Lagrangian. It is
proved by \cite{[S1]} that the property of Lagrangian is preserved along the mean curvature flow. Thus,
it is possible to use the flow method to construct minimal Lagrangian submanifolds. A natural question is how
to analyze the long time behavior or singularities along the Lagrangian mean curvature flow. In \cite{[TY]}, Thomas-Yau conjectured that under
some stability conditions the Lagrangian mean curvature flow exists for all time and converges to a special
Lagrangian submanifold in its hamiltonian deformation class. There are several
results  relevant to this conjecture. In \cite{[S3]}\cite{[SW]} Smoczyk and Smoczyk-Wang
  proved the long time existence and convergence of the Lagrangian mean curvature flow into a flat space
 under some convexity conditions respectively, and in \cite{[ACH]} Chau-Chen-He
 studied the flow of entire Lagrangian graph with Lipschitz continuous initial data. In \cite{[W]}, M. T. Wang also proved the convergence  for the graph
of a symplectomorphism between Riemann surfaces. However, the flow will develop finite time singularities in
general, and the readers are referred to \cite{[W0]} \cite{[CL3]}\cite{[N1]}\cite{[N2]}\cite{[HL2]} and
references therein.

\medskip

In this paper, we will consider the Lagrangian mean curvature flow in a
general K\"ahler-Einstein manifold with arbitrary
dimension under some stability conditions. Let $(M, \bar g)$ be a complete K\"ahler-Einstein manifold
 with scalar curvature $\bar R$ and
\beq K_5=\sum_{i=0}^5\sup_{\bar M}\;|\bar \Na^i\bar Rm|<\infty, \quad inj(M)\geq \iota_0>0,\label{eq:00}\eeq
where $inj(M)$ is the injectivity radius of $(M, \bar g).$
The first main result is

\begin{theo}\label{main}Let $(M, \bar g)$ be a complete K\"ahler-Einstein manifold satisfying (\ref{eq:00})
 with scalar curvature $\bar R<0$,
and $L$ be a compact Lagrangian submanifold smoothly immersed in
$M$. For any $ V_0, \La_0>0$, there exists $\ee_0=\ee_0(V_0, \La_0,
\bar R, K_5, \iota_0)>0$ such that if $L$ satisfies \beq \vol(L)\leq
V_0, \quad |A|\leq \La_0,\quad \int_L\;|H|^2\leq
\ee_0,\label{main01:eq1}\eeq where $A$ is the second fundamental
form of $L$ in $M$ and $H$ is the mean curvature vector,
 then the Lagrangian mean  curvature flow with the initial data $L$ will converge exponentially fast to a minimal
Lagrangian submanifold in $M.$
\end{theo}

\medskip

Here we need to assume the scalar curvature of the ambient
K\"ahler-Einstein manifold is negative because any minimal
Lagrangian submanifold is strictly stable in this situation. Thus,
it is natural to expect that for any small perturbation of a minimal
Lagrangian submanifold, the Lagrangian mean curvature flow will
exist for all time and deform it to a minimal Lagrangian
submanifold. Theorem \ref{main} shows that this is indeed true, but
we don't need to assume the existence of  minimal Lagrangian
submanifolds.
\medskip

For a K\"ahler-Einstein manifold with nonnegative scalar curvature, we have the result:

\begin{theo}\label{main2}
Let $(M, \bar g)$ be a complete K\"ahler-Einstein manifold satisfying (\ref{eq:00}) with scalar
curvature $\bar R\geq 0$, and $L$ be a compact Lagrangian
submanifold smoothly immersed in $M$. For any $ V_0,
\La_0, \dd_0>0$, there exists $\ee_0=\ee_0(V_0,
\La_0,  \dd_0, \bar R,  K_5, \iota_0)>0$ such that if
\begin{enumerate}
\item  the mean curvature form of $L$ is exact,
    \item $L$ satisfies
  $$\la_1\geq \frac {\bar R}{2n}+\dd_0,\quad\vol(V)\leq V_0, \quad |A|\leq \La_0,\quad  \int_{L}\,|H|^2\leq \ee_0,$$
  where $\la_1$ is the first eigenvalue of the Laplacian operator with respect to the induced metric on $L,$
\end{enumerate}
 then the Lagrangian mean  curvature flow with the initial data $L$ will converge exponentially fast to a minimal
Lagrangian submanifold in $M.$

\end{theo}

For Lagrangian submanifolds in K\"ahler-Einstein manifolds with
positive scalar curvature, a notion of hamiltonian stability was
introduced in \cite{[Oh]} to characterize the variations of the
submanifold under hamiltonian deformations. The hamiltonian
stability is more natural than the standard stability for the case
when the scalar curvature is positive. For example, $\RR\PP^n$ and
the Clifford torus $\TT^n $ in $\CC\PP^n$ are hamiltonian stable but
not (Lagrangian) stable in the standard sense. Thus, to get a
convergence result for the Lagrangian mean curvature flow it is
natural to expect that the deformation along the flow is
hamiltonian, which is equivalent to say that the mean curvature form
along the flow is exact. Fortunately, the  exactness of the mean
curvature form is preserved along the mean curvature flow. This is
why we need the assumption $1$ in Theorem \ref{main2}.

\medskip

In \cite{[Oh]}, Y. G. Oh proved that a minimal Lagrangian submanifold is hamiltonian stable if and only if
the first eigenvalue of the Laplacian operator $\la_1\geq {\bar R}/{2n}.$ Thus, the assumption 3 of
Theorem \ref{main2} on the first eigenvalue ensures that the limit minimal Lagrangian submanifold is strictly
hamiltonian stable. Since for the well-known examples  $\RR\PP^n$ and the Clifford torus $\TT^n $ in $\CC\PP^n$
the first eigenvalue $\la_1={\bar R}/{2n}$, we can see that Theorem \ref{main2} cannot be applied.  It is interesting to know whether we have the corresponding
result in this situation. This phenomenon is similar to the K\"ahler-Ricci flow on K\"ahler manifolds with nonzero holomorphic vector fields  (cf. \cite{[CL]}).

\medskip

Before stating the third result, we introduce

\begin{defi}
A vector field $X$ is called an essential hamiltonian variation of $L$ , if
$X$ can be written as $X=J\Na f$ where $f\notin E_{\la_1}$. Here $E_{\la_1}$ is the first eigenspace
of the Laplacian operator $\Delta$ on $L$.
\end{defi}

\begin{theo}\label{main3}Let $(M, \bar g)$ be a compact K\"ahler-Einstein manifold with
$\bar R\geq 0.$ Suppose that $\phi: L\ri M$ is a compact minimal Lagrangian
submanifold with the first eigenvalue $ {\bar R}/{2n}$ and $X$ is an essential hamiltonian variation of
$L_0=\phi(L)$. Let
$\phi_s: L\ri M(s\in (-\eta, \eta))$ with $\phi_0=\phi$ be a
one-parameter family of hamiltonian deformations generated by $X$.
Then there exists
$\ee_0=\ee_0(X, L_0, M)>0$ such that if
$L_s=\phi_s(L)\subset M$ satisfying
$$\|\phi_s-\phi_0\|_{C^3}\leq\ee_0, $$
then the Lagrangian mean curvature flow with the initial Lagrangian
submanifold $L_s$ will converge exponentially fast
to a minimal Lagrangian submanifold in $M.$
\end{theo}

\medskip
Note that we can show that a minimal Lagrangian submanifold $L$ with $\la_1=\bar R/2n$ is
 strictly hamiltonian stable along an essential hamiltonian variation $X$(cf. Lemma \ref{lem:positive}).
Theorem \ref{main3} says that the flow will exist for all time and converge if the initial Lagrangian
submanifold is a small perturbation of $L$ along  essential hamiltonian variations, which is
reasonable since $L$ is strictly hamiltonian stable along these directions.

\medskip

The idea of the proofs of Theorem \ref{main} and \ref{main2} is similar to that used in \cite{[CL]}. First, we
use the smallness of the mean curvature vector in a short time interval to get the exponential decay of
the $L^2$ norm of the mean curvature vector, which is a crucial step in the whole argument. Then, by a
simple observation(cf. Lemma \ref{lem:vol}) we can
get all higher order estimate of the second fundamental form from the decay of the $L^2$ norm of
the mean curvature. This step relies on the noncollapsing assumption
of Lagrangian submanifolds, which is a technical condition and can be removed in the proof of
 the main theorems. This step is different from the K\"ahler-Ricci flow in \cite{[CL]}\cite{[CLW]}, where we
  use the parabolic Moser iteration to get $C^0$ order estimate of the K\"ahler potential.
 Then, we can show that the exponential decay of the mean curvature vector implies
that the second fundamental form is uniformly bounded for any time interval and we can extend the
solution for all time. The readers are referred to \cite{[CL]} for more details of the argument.
\medskip

\medskip

In a forthcoming paper, we expect to extend the argument in the present paper to the case when the initial
submanifold is not Lagrangian. Our argument might be also useful for the symplectic mean curvature flow
(cf. \cite{[CL2]}\cite{[HL]}), and we will explore this in the future.

\medskip

This paper is organized as follows: In Section 2, we recall some basic facts and evolution equations of
mean curvature flow and Lagrangian submanifolds. In particular, we will give some details of the proof which will be used in the paper.
In Section 3, we will show several technical lemmas along the Lagrangian mean curvature flow. In
Section 4 and 5, we will finish the proof of Theorem \ref{main} and \ref{main2}. In Section 6, we will
recall some basic facts on the deformation of minimal Lagrangian submanifolds and finish the proof of
Theorem \ref{main3}. In the last section, we collect some examples where our theorems can be applied.

\bigskip

{\bf Acknowledgements}: The author  would like to thank  Professor X. X. Chen, W. Y. Ding and F. Pacard for their constant,
warm encouragements over the past several years. We would also like to thank W. Y. He for numerous suggestions
which helped to improve the whole paper.

\section{Notations and preliminaries}
In this section, we recall some evolution equations from \cite{[CL2]}
for the mean curvature flow in arbitrary dimension and codimension,
and then we discuss  the special case of Lagrangian mean
curvature flow.

Let $(M, \bar g)$ be a $m$-dimensional Riemannian manifold and $F_0:
L\ri M$ be a smoothly immersed submanifold with dimension $n.$ We
consider the a one-parameter family of smooth maps $F_t: L\ri M$
with the image $L_t=F_t(L)$ smooth submanifold in $M$ and $F$
satisfies \beq \pd {}tF(x, t)=H(x, t),\quad F(x,
0)=F_0(x).\label{MCF}\eeq Here $H(x, t)$ is the mean curvature
vector of $L_t$ at $F(x, t)$ in $M.$ Choose a local orthonormal
frame $e_1, \cdots, e_n, e_{n+1}, \cdots, e_m$ of $M$ along $L_t$
such that $e_1, \cdots, e_n$ are tangent vectors of $L_t$ and
$e_{n+1}, \cdots, e_m$ are in the normal bundle over $L_t.$ The
second fundamental form and the mean curvature operator are given by
$$A=A^{\al}e_{\al},\quad H=-H^{\al}e_{\al},$$
where $\al\in \{n+1, \cdots, m\}$. Let $A^{\al}=(h^{\al}_{ij})$
where $(h_{ij}^{\al})$ is a matrix given by
$$h_{ij}^{\al}=\bar g(\bar \Na_{e_i} e_{\al}, e_j)=\bar g(\bar \Na_{e_j}e_{\al}, e_i)=h_{ji}^{\al}$$
where $\bar \Na$ is the Levi-Civita connection on $M.$ The mean
curvature $H^{\al}=g^{ij}h_{ij}^{\al},$ where $g_{ij}=\bar g(e_i,
e_j)$ is the induced metric on $L$. By direct calculation we have
the evolution equation of the induced metric
$$\pd {}tg_{ij}=-2H^{\al}h^{\al}_{ij}.$$

\medskip

With these notations, we have the evolution equations of the second fundamental form and the mean
curvature vector.

\begin{lem}\label{lem:AH}(cf. \cite{[CL2]}) The evolution equation of the second fundamental form is
given by \beqn \pd {}th_{ij}^{\al}
&=&\Na_i\Na_jH^{\al}-H^{\ga}h_{jl}^{\ga}h_{il}^{\al}+
h_{ij}^{\bb}b_{\bb}^{\al}+H^{\bb}\bar R_{\al j\bb
i}\label{lem:AH:eq6}
\\&=&\Delta
h_{ij}^{\al}+h_{il}^{\bb}h_{ml}^{\bb}h_{mj}^{\al}-H^{\bb}(h_{mi}^{\bb}h_{mj}^{\al}+
h_{mj}^{\bb}h_{mi}^{\al})+h_{ij}^{\bb}h_{ml}^{\bb}h_{ml}^{\al}-h_{im}^{\bb}h_{jl}^{\bb}h_{ml}^{\al}\nonumber
\\
&&-(\bar \Na_l \bar R_{\al jil}+\bar \Na_i \bar R_{\al ljl})-(\bar
R_{\bb\al jl}h_{il}^{\bb}+\bar R_{\al\bb il}h_{jl}^{\bb}) +(\bar
R_{mllj}h_{im}^{\al}+\bar R_{illm}h_{jm}^{\al})\nonumber
\\&&+2\bar
R_{iljm}h_{ml}^{\al}+\bar R_{\al l\bb l}h_{ij}^{\bb}+
h_{ij}^{\bb}b_{\bb}^{\al}, \label{lem:AH:eq5}\eeqn
where $b_{\al}^{\bb}=\bar g(\pd {}{t}e_{\al}, e_{\bb}).$  The equation of
mean curvature vector is given by \beq \pd {}tH^{\al}=\Delta
H^{\al}+H^{\bb}h_{lm}^{\bb}h_{ml}^{\al}+H^{\bb}\bar R_{\al l \bb
l}+H^{\bb}b_{\bb}^{\al}.\label{eq:H}\eeq Here $\bar R_{ABCD}$ is the
curvature tensor in $M$ and we choose the convention such that $\bar
R_{uvuv}>0$ for round spheres.
\end{lem}
\begin{proof}The equations (\ref{lem:AH:eq6})(\ref{lem:AH:eq5}) follow directly from Lemma 2.3-2.5 and
Proposition 2.6 in \cite{[CL2]}, and (\ref{eq:H}) follows from (\ref{lem:AH:eq6}) and the definition of $H.$

\end{proof}

\bigskip

Now we recall some basic facts of Lagrangian mean curvature flow
from \cite{[S1]}-\cite{[S4]}. Assume that $(M, \bar g, J)$ is a
K\"ahler-Einstein manifold of real dimension $m=2n,$ and $L$ is an
$n$-dimensional manifold smoothly immersed into $M$ by a smooth map
$F: L\ri M.$ Let $\bar \oo$ be the associate K\"ahler form of the metric $\bar
g$. The submanifold $L_0=F(L)\subset M$ is called Lagrangian if
$$F^*\bar \oo=0,\quad \hbox{on } L.$$ Choose normal coordinates
$\{x^i\}$ for $L$ and we have that $e_i=\p_iF$  are the tangent
vectors of $L.$ Since $L$ is Lagrangian, $Je_i$ is a normal vector
for any $i=1, \cdots, n.$ In fact,
$$\bar g(Je_i, e_j)=\bar \oo(e_i, e_j)=0.$$
Hence, $\{e_i, Je_i\}$ is a local coordinate frame of $M.$ For
convenience, we use that an underlined index denotes the application
of the complex structure $J$. For example
$$\bar g_{i\un j}=\bar g(e_i, Je_j).$$
For simplicity, we denote by $h^{k}_{ij}$  the second fundamental
form $h^{\un k}_{ij}=-\bar g(Je_k, \bar \Na_{e_i}e_j)$. Since $L$ is
Lagrangian, it is easy to check that the second fundamental form has
full symmetry
$$h^{k}_{ij}=h^{k}_{ji}=h^{i}_{kj}.$$
The mean curvature vector $H=-H^iJe_i$ where $H^i=g^{kl}h^i_{kl}.$
The norms of
the second fundamental form and the mean curvature vector are given by
$$|A|^2=h^k_{ij}h^{l}_{pq}g^{ip}g^{jq}g_{kl},\quad |H|^2=H^iH^jg_{ij}.$$
We define the mean curvature form by $\al_H=g_{ij}H^jdx^i$, and we have the
following well-known result:

\begin{lem}If $M$ is a K\"ahler-Einstein manifold, then $\al_H$ is a closed
1-form. \end{lem}
\begin{proof}By the full symmetry of $h^k_{ij}$, we have
 \beq d\al_H(e_i, e_j)=\Na_iH^j-\Na_jH^i\n=\Na_i h^k_{kj}-\Na_jh^k_{ki}=\bar R_{jik\un k}, \label{eq:H1}\eeq
where we have used the Codazzi equation
$$\Na_i h^l_{jk}-\Na_jh_{ik}^l=-\bar R_{\un lkji}.$$
Since $M$ is a K\"ahler-Einstein manifold, by equality (\ref{eq:H2}) below we have
$$d\al_H(e_i, e_j)=\bar R_{\un j i}=\frac {\bar R}{2n}\bar \oo(e_j, e_i)=0,\quad \hbox{on}\; L,$$
since $L$ is Lagrangian. The lemma is proved.
\end{proof}

\medskip

For the mean curvature flow (\ref{MCF}), if the initial data $L_0$
is Lagrangian, then the submanifolds $L_t$ are all Lagrangian (cf. \cite{[S1]}). Thus, we call the flow (\ref{MCF}) the Lagrangian mean
curvature flow if the initial submanifold $L_0$ is Lagrangian. It was
proved by \cite{[S4]} that the exactness of
the mean curvature form of $L_t$ is preserved along the Lagrangian
mean curvature flow. \medskip

\begin{lem}\label{lem:LH}Along the Lagrangian mean curvature flow, we have
\beq
 \pd {}tH^{i}=\Delta H^{i}+H^{j}h_{lm}^{j}h_{ml}^{i}+H^{j}\bar R_{\un i l \un j l}-H^{j}H^kh^i_{kj}.\label{eq:LH}
\eeq
The second fundamental form satisfies
\beq
\pd {}th_{ij}^{k}
=\Na_i\Na_jH^{k}-H^{m}h_{jl}^{m}h_{il}^{k}-H^lh^l_{mk}
h_{ij}^{m}+H^{m}\bar R_{\un k j\un m
i}.\quad \label{eq:LH2}
\eeq
\end{lem}
\begin{proof}Since $L_t$ is Lagrangian along the flow, we have
\beq b^{\un i}_{\un j}=\bar g(\pd {}te_{\un i}, e_{\un j})=\bar g(J\Na_{e_i}(-H^ke_{\un k}), e_{\un j})=-H^kh^{k}_{ij}.\label{eq:LH3}\eeq
By (\ref{eq:H}), we have \beq
 \pd {}tH^{i}=\Delta H^{i}+H^{j}h_{lm}^{j}h_{ml}^{i}+H^{j}\bar R_{\un i l \un j
 l}-H^{j}H^kh^k_{ij}.\n
\eeq
The equation (\ref{eq:LH2}) follows directly from (\ref{eq:LH3}) and Lemma \ref{lem:AH}.

\end{proof}
\medskip

\begin{lem}\label{lem:angle}(cf. \cite{[S1]}) If the initial mean curvature form is exact, then there
exists a smooth angle function $\theta(x, t)$ such that
$\al_H=d\theta$ and \beq \pd {\theta}t=\Delta \theta+\frac {\bar
R}{2n}\theta. \label{eq:angle}\eeq

\end{lem}
\begin{proof}It follows from \cite{[S1]} that $\al_H(t)$ are exact as long as the solution exists if the initial mean curvature form is exact. Since $H^i=\Na^i\theta,$ and we calculate
\beqn
\Delta \Na^i \theta&=&\Na_k\Na_i\Na^k\theta=(\Na_i\Na_k\Na^k\theta+R_{kikl}\Na^l\theta)\n\\
&=&\Na^i\Delta\theta+(\bar
R_{kikl}+H^mh_{il}^m-h^m_{kl}h^m_{ki})\Na^l\theta.\label{lem:angle:eq1} \eeqn
Combining (\ref{lem:angle:eq1}) with (\ref{eq:LH}), we have
 \beqs \Na^i\pd {\theta}t&=&\Na^i\Delta \theta+(\bar R_{\un im\un l m}+\bar R_{miml})\Na^l\theta\\
 &=&\Na^i\Delta \theta+\frac {\bar R}{2n}\Na^i\theta.\eeqs
Hence, (\ref{eq:angle}) is proved.

\end{proof}

\medskip
For the readers'convenience, we collect some basic facts on curvatures in a K\"ahler manifold.
Let $(M, \bar g)$ be a K\"ahler manifold, the Ricci
curvature is given by
$$\bar R_{AC}=\bar g^{BD}\bar R_{ABCD}.$$
Here doubled latin capitals are summed from $1$ to $2n.$ Now in the
local frame $\{e_i, Je_i\}$ we calculate \beqs \bar
R_{AB}&=&\bar R(e_A, e_k, e_B, e_k)+\bar R(e_A, Je_k, e_B, Je_k)\\
&=&\bar R(e_A, e_k, Je_B, Je_k)-\bar R(e_A, Je_k, Je_B, e_k)\\
&=&\bar R(e_k, e_A, Je_k, Je_B)+\bar R(e_A, Je_k, e_k, Je_B)\\
&=&-\bar R(Je_k, e_k, e_A, Je_B)=\bar R_{A\un Bk\un k}.\eeqs Hence,
we have
\beq \bar R_{AB}=\bar R_{A\un Bk\un k},\quad \bar R_{A\un
B}=-\bar R_{ABk\un k}.\label{eq:H2}\eeq The scalar curvature $\bar R=\bar
R_{kk}+\bar R_{\un k\un k}=2\bar R_{kk}.$ Since $\bar g$ is a
K\"ahler-Einstein metric, we have $\bar R_{ij}=\frac {\bar
R}{2n}\bar g_{ij}.$

\section{Estimates}
In this section, we derive some estimates along Lagrangian mean
curvature flow.

\subsection{The mean curvature vector}\label{sec3.1}
In this subsection, we will prove that the $L^2$ norm of the mean
curvature vector decays exponentially under certain conditions. More
precisely, we will prove that the $L^2$ norm of the mean curvature
vector will decays exponentially when the mean curvature is small
and the scalar curvature of the ambient K\"ahler-Einstein manifold
is negative. For the case of nonnegative scalar curvature, an
interesting condition on the exactness of the mean curvature form is
assumed to ensure the exponential decay when the mean curvature is
small.

\begin{lem}\label{Lem:H}
Let $(M, \bar g)$ be a K\"ahler-Einstein manifold with
scalar curvature $\bar R.$  For any $\La, \ee>0$,  if the solution $L_t(t\in [0, T])$ of Lagrangian mean curvature flow satisfies
$$|A|(t)\leq \La, \quad |H|(t)\leq \ee,\quad t\in [0, T],$$
then we have the inequality
 \beq \frac d{dt}\int_{L_t}\; |H|^2d\mu_t\leq \Big(\frac 1n{\bar R}+2\La\ee\Big)\int_{L_t}\; |H|^2d\mu_t,\quad
 t\in [0, T]. \label{lem:H:eq1}\eeq
 Moreover, if we assume that the mean curvature form of $L_0$ is
 exact, then
\beqn
\pd
{}t\int_{L_t}\;|H|^2d\mu_t
&\leq& -2\Big(\la_1-\frac {\bar R}{2n}-\La\ee\Big)\int_{L_t}\;|H|^2d\mu_t, \label{lem:H:eq3}
\eeqn
where $\la_1$ is the first eigenvalue of $\Delta$ with respect to the induced metric on $L.$

\end{lem}
\begin{proof}
By (\ref{eq:LH}) we calculate \beqn &&\pd {}t\int_L\;|H|^2d\mu_t\n\\&=&\int_{L_t}\;\pd {g_{ij}}tH^iH^j
+2H^i\pd {}{t}H^i-|H|^4\n\\
&=&\int_L\; 2H^i\Delta H^i+2H^iH^jh_{kl}^jh_{lk}^i+2H^iH^{j}\bar R_{\un i m\un j m}-4H^iH^jH^kh_{jk}^i-|H|^4\nonumber\\
&\leq&\int_L\;-2|\Na_kH^i|^2+2H^iH^jh_{kl}^jh_{lk}^i+2H^iH^{j}\bar R_{\un i m\un j m}-4H^iH^jH^kh_{jk}^i-|H|^4\label{lem:AH:eq1}\eeqn

We claim that for any vector field $X=X^ie_i$ on $L$,  the inequality holds
\beq
\int_L\;|\Na_iX^k|^2-h_{km}^lh_{km}^iX^iX^l\geq\int_L\;|\Na_i X^i|^2-H^mh_{il}^mX^iX^l-\bar R_{kikl}X^iX^l.\label{lem:AH:eq4}
\eeq
In fact,
\beqn
0&\leq&\frac 12\sum_{i, k}\int_L\; |\Na_i X^k-\Na_k X^i|^2=\sum_{i, k}\int_L\;|\Na_iX^k|^2-
g(\Na_i X^k, \Na_k X^i)\nonumber\\
&=&\sum_{i, k}\int_L\;|\Na_iX^k|^2+g(\Na_k\Na_iX^k, X^i).\label{lem:AH:eq2}
\eeqn
Note that we can change the covariant derivatives
\beqn
\Na_k\Na_iX^k&=&\Na_i\Na_kX^k+R_{kikl}X^l\nonumber\\
&=&\Na_i\Na_kX^k+(\bar R_{kikl}+h^m_{kk}h^m_{il}-h^m_{kl}h^m_{ik})X^l,\label{lem:AH:eq3}
\eeqn
where $R_{ijkl}$ is the curvature tensor on $L$ and we used the Gauss equation
$$R_{ijkl}=\bar R_{ijkl}+h^{\al}_{ik}h^{\al}_{jl}-h^{\al}_{il}h^{\al}_{jk}.$$
Thus, (\ref{lem:AH:eq2}) and (\ref{lem:AH:eq3}) imply that
\beqs
0&\leq&\int_L\;|\Na_iX^k|^2-|\Na_i X^i|^2+\bar R_{kikl}X^iX^l+H^mh_{il}^mX^iX^l-h_{km}^lh_{km}^iX^iX^l,
\eeqs
which proves (\ref{lem:AH:eq4}).

\medskip
Now we apply the inequality (\ref{lem:AH:eq4}) for the vector
$H^ie_i$ and combine this with (\ref{lem:AH:eq1})
\beqn &&\pd
{}t\int_{L_t}\;|H|^2d\mu_t\n\\
&\leq&\int_{L_t}\;-2|\Na_kH^i|^2+2H^iH^jh_{kl}^jh_{lk}^i+2H^iH^{j}\bar R_{\un i m\un j m}-4H^iH^jH^kh_{jk}^i-|H|^4\n\\
&\leq&\int_{L_t}\;-2|\Na_i H^i|^2+2(\bar R_{kikl}H^iH^l+H^iH^{j}\bar R_{\un i m\un j m})-2H^iH^jH^kh_{jk}^i-|H|^4.\label{lem:H:eq2}
\eeqn
Note that
$\bar R_{kikl}+\bar R_{\bar i k\bar lk}=\bar R_{il},$
by the assumption we have
\beqn \pd
{}t\int_{L_t}\;|H|^2d\mu_t&\leq&\int_{L_t}\;2\bar R_{ij}H^iH^j-2H^iH^jH^kh_{jk}^i\n\\
&\leq&\Big(\frac {\bar R}n+2\La\ee\Big)\int_{L_t}\;|H|^2d\mu_t.
\n\eeqn Thus, (\ref{lem:H:eq1}) is proved.
\medskip

If we assume that the mean curvature form of $L_0$ is exact, then by Lemma \ref{lem:angle} the mean curvature form is also
exact for all $t.$ Thus, there exists a smooth function $\theta(x, t)$ with $H^i=\Na^i\theta,$ which implies
$$\int_{L_t}|\Na_iH^i|^2= \int_{L_t}|\Delta \theta|^2\geq \la_1\int_{L_t}|\Na \theta|^2=\la_1\int_{L_t}|H|^2.$$
Combining this with (\ref{lem:H:eq2}), we have
\beqn
\pd
{}t\int_{L_t}\;|H|^2d\mu_t
&\leq& -2\Big(\la_1-\frac {\bar R}{2n}-\La\ee\Big)\int_{L_t}\;|H|^2d\mu_t.\n
\eeqn
Thus, (\ref{lem:H:eq3}) is proved.
\end{proof}

\subsection{The first eigenvalue}
In previous section, we know that when the scalar curvature of the
ambient manifold is nonnegative, the exponential decay of the $L^2$
norm of mean curvature vector will depend on the behavior of the
first eigenvalue of the Laplacian along the flow.  In this
subsection, we give some estimates on the first eigenvalue, which
essentially says that the first eigenvalue will have a positive
lower bound if the mean curvature vector decays exponentially.

\begin{lem}\label{lem:eig}Along the Lagrangian mean curvature flow, we have
\begin{enumerate}
  \item For any constants $\dd, \La>0,$ there exists $t_0=t_0( n, \La, K_2, \dd)$ such that if the solution $L_t$ satisfies
  $|A|\leq \La$ for $t\in [0, t_0]$, then
  \beq
\sqrt{\la_1(t)}\geq \sqrt{\la_1(0)}(1-\dd)-\dd,\quad t\in [0, t_0]. \label{lem:eig:eq3}
  \eeq
  \item For any constants $T, \ee, \ga, \La>0$, if the solution $L_t$ satisfies
  \beq
|A|\leq \La, \quad |\Na H|+|H|\leq \ee e^{-\ga t},\quad t\in [0, T], \label{lem:eig:eq4}
  \eeq then we have the estimate
  \beq
\sqrt{\la_1(t)}\geq \sqrt{\la_1(0)}e^{-\frac 1{2\ga}(2\La \ee+\ee^2)}-\frac {(K_0+\La)\ee}{\ga},\quad t\in [0, T].
\label{lem:eig:eq5}
  \eeq
\end{enumerate}

\end{lem}

\begin{proof}Let $f(x, t)$ be a eigenfunction of the Laplacian operator with respect to the induced metric on $L_t$ satisfying
$$-\Delta f=\la_1(t) f,\quad \int_{L_t}f^2d\mu_t=1.$$
Taking derivative with respect to $t$, we have
\beq
\int_{L_t}\;2f\pd ft-f^2|H|^2=0.\label{lem:eig:eq6}
\eeq
Observe that the first eigenvalue satisfies
$$\la_1(t)=\int_{L_t}|\Na f|^2d\mu_t.$$
Thus, we calculate
\beqn
\pd {\la_1}t&=&-\pd {}t\int_{L_t}\;f\Delta f d\mu_t\n\\
&=&-\int_{L_t}\;2\pd ft\Delta f+f\Big(\pd {}t\Delta\Big)f-f\Delta f|H|^2\n\\
&=&\int_{L_t}\;\la_1\Big(2f\pd ft-f^2|H|^2\Big)-2H^ih^i_{kl}f\Na_k\Na_lf\n\\
&=&\int_{L_t}\;-2H^ih^i_{kl}f\Na_k\Na_lf\n\\
&=&\int_{L_t}\;2H^ih^i_{kl}\Na_kf\Na_lf+2(H^ih^i_{kl})_kf\Na_lf,\label{lem:eig:eq1}
\eeqn where we used the equality (\ref{lem:eig:eq6}). Note that by the Codazzi equation we have
$$\Na_kh^i_{kl}-\Na_lh^i_{kk}=-\bar R_{\un iklk}.$$
Combining this with (\ref{lem:eig:eq1}), we have
\beqn
\pd {\la_1}t&=&\int_{L_t}\;2H^ih^i_{kl}\Na_kf\Na_lf+2(H^ih^i_{kl})_kf\Na_lf\n\\
&=&\int_{L_t}\;2H^ih^i_{kl}\Na_kf\Na_lf+2(\Na_kH^ih^i_{kl}+H^i\Na_lH^i-H^i\bar R_{\un iklk })f\Na_lf\n\\
&=&\int_{L_t}\;2H^ih^i_{kl}\Na_kf\Na_lf-|H|^2(|\Na f|^2+f\Delta f)\n\\&&-2H^i\bar R_{\un iklk }f\Na_lf
+2\Na_kH^ih^i_{kl}f\Na_lf.\label{lem:eig:eq7}
\eeqn
(1). Under the assumption $1$, by Lemma \ref{lem:zero} and \ref{lem:higer} there exist  positive constants
$\un t=\un t(n, \La, K_1)$ and $a_1=a_1(n, \La, K_2)$
 such that
$$|\Na H|\leq \frac {a_1}{\sqrt{t}},\quad t\in (0, \un t].$$
Thus, by (\ref{lem:eig:eq7}) we have
\beqn
\pd {\la_1}t&\geq&-4\La^2\la_1-2K_0\La \la_1^{\frac 12}-\frac {2\La a_1}{\sqrt{t}}\la_1^{\frac 12}\n\\
&=&-c_1\la_1-\Big(\frac {c_2}{\sqrt{t}}+c_3\Big)\la_1^{\frac 12}, \quad t\in (0, \un t]\n
\eeqn where $c_1=4\La^2, c_2=2\La a_1$ and $c_3=2K_0\La.$ Thus, we have
\beq
\sqrt{\la_1(t)}\geq \sqrt{\la_1(0)} e^{-\frac {c_1}2t}-\frac {c_3}2t-c_2\sqrt{t},\quad t\in [0, \un t].\n
\eeq
If we choose $t$ sufficiently small, then (\ref{lem:eig:eq3}) is proved.

\medskip

(2). Under the assumption (\ref{lem:eig:eq4}), by (\ref{lem:eig:eq7}) we have
\beqn
\pd {\la_1}t&\geq&-2\La \ee e^{-\ga t}\la_1-2\ee^2 e^{-2\ga t}\la_1-2K_0 \ee
e^{-\ga t}\la_1^{\frac 12}-2\La \ee e^{-\ga t}\la_1^{\frac 12}\n\\
&\geq&-(2\La \ee e^{-\ga t}+2\ee^2 e^{-2\ga t})\la_1-2(K_0+\La) \ee e^{-\ga t}\la_1^{\frac 12},\quad t\in [0, T].\n
\eeqn
Thus, we have
\beq
\sqrt{\la_1(t)}\geq \sqrt{\la_1(0)}e^{-\frac 1{2\ga}(2\La \ee+\ee^2)}-\frac {(K_0+\La)\ee}{\ga}.\n
\eeq

\end{proof}
\subsection{Zero order estimates}
In Section \ref{sec3.1}, we proved the exponential decay of the
$L^2$ norm of the mean curvature vector under some conditions. To
get a pointwise decay of the mean curvature form, we need to do more
work. One way is to use the parabolic Moser iteration as in the
K\"ahler-Ricci flow in \cite{[CLW]} and \cite{[CL]}. However, the
Sobolev inequality for submanifolds in a Riemannian manifold needs
many restrictions (cf. \cite{[HS]}). Here we give a simple
observation to bound the $C^0$ estimates by the $L^2$ norm. First,
we  introduce the following definition, which is inspired by Ricci
flow \cite{[Pe]}:

\begin{defi}
A geodesic ball $B(p, \rho)\subset L$ is called $\kappa$-noncollapsed if $\vol(B(q, s))\geq \kappa s^n$ whenever
$B(q, s)\subset B(p, \rho).$ Here the volume is with respect to the induced metric on $L.$ A Riemannian manifold $L$
 is called $\kappa$-noncollapsed on the scale $r$ if every geodesic ball $B(p, s)$ is $\kappa$-noncollapsed for $s\leq r.$
\end{defi}

\begin{lem}\label{lem:vol}If $L_0$ is $\kappa_0$-noncollapsed
on the scale $r_0$, then for any  small geodesic ball $B_{t}(p, \rho)$ in $L_t$ with radius $\rho\in (0, r_0)$, we have
$$\vol(B_{t}(p, \rho))\geq \kappa_0 e^{-(n+1)E(t)}\rho^n,$$
where $E(t)$ is given by
\beq E(t)=\int_0^t\;\max_{L_s}(|A||H|+|H|^2)\;ds.\label{lem:vol:eq1}\eeq
\end{lem}
\begin{proof}Recall that the evolution equation of the induced metric on $L$ is given by
$$\pd {}tg_{ij}=-2H^{k}h^k_{ij},$$
which implies that the distance function satisfies
$$e^{-E(t)}d_0(x, y)\leq d_t(p, q)\leq d_0(p, q)e^{E(t)}$$
 and the volume form
 $d\mu_t\geq e^{-E(t)}d\mu_0,$
 where $E(t)$ is given by (\ref{lem:vol:eq1}). Thus, the volume of $B_t(p, \rho)$ has the estimate
\beqs
\vol(B_{p}(\rho))&=& \int_{B_t(p, \rho)}\; d\mu_t\geq \int_{B_0(p, e^{-E(t)}\rho)}e^{-E(t)}d\mu_0\geq \kappa_0 e^{-(n+1)E(t)}\rho^n,
\eeqs
as long as $\rho\leq r_0$ since $L_0$ is $\kappa_0$-noncollapsed on the scale $r_0.$ The lemma is proved.

\end{proof}

To derive the zero order estimate of the mean curvature vector, we prove the following simple result:
\begin{lem}\label{lem:f}Suppose that $L$ is $\kappa$-noncollapsed
on the scale $r$. For any tensor $S$ on $L$, if
$$|\Na S|\leq \La, \quad \int_L\; |S|^2\;d\mu\leq \ee,$$
where $\ee\leq r^{n+2},$ then
$$\max_L|S|\leq \Big(\frac 1{\sqrt{\kappa}}+\La\Big)\ee^{\frac 1{n+2}}.$$
\end{lem}
\begin{proof}
Assume that $|S|$ attains its maximum at point $x_0\in L$. Thus, for any point $x\in B(x_0, \dd) $ with
small $\dd>0$ we have
$$|S(x)|\geq |S(x_0)|-\La \dd>0.$$
Hence, we have the inequality
$$\ee\geq \int_{B(x_0, \dd)}|S|^2\,d\mu\geq (|S(x_0)|-\La \dd)^2\vol(B(x_0, \dd))\geq(|S(x_0)|-\La \dd)^2\kappa \dd^n. $$
Let $\dd=\ee^{\frac 1{n+2}}$ and we choose $\ee$ small such that $\ee^{\frac 1{n+2}}\leq r$, then
$$\max_L|S|\leq \Big(\frac 1{\sqrt{\kappa}}+\La\Big)\ee^{\frac 1{n+2}}.$$
The lemma is proved.
\end{proof}

\subsection{Higher order estimates}
In this subsection, we collect some basic estimates for the second
fundamental form, which can be proved by the maximum principle. The
following result shows that the second fundamental form doesn't
change too much near the initial time.

\begin{lem}\label{lem:zero}
Along the Lagrangian mean curvature flow, if $L_0$ satisfies
$$|A|(0)\leq \La, \quad |H|(0)\leq \ee, $$
then there exists $T=T(n, \La, K_1)$ such that $L_t$ has the estimates
\beq |A|(t)\leq 2\La, \quad |H|(t)\leq 2\ee, \quad t\in [0, T].\label{lem:zero:eq1}\eeq

\end{lem}
\begin{proof}It follows from the maximum principle. Recall that by (\ref{lem:AH:eq5}) the second fundamental form satisfies
$$\pd {}t|A|\leq \Delta|A|+c_1(n)|A|^3+c_2(n, K_0)|A|+c_3(n, K_1).$$
Let $t_0=\sup\{s>0\;|\;|A|(t)\leq 2\La,\;  t\in [0, s)\}$. Then, for $t\in [0, t_0)$ we have the inequality
$$\pd {}t|A|\leq \Delta |A|+8\La^3c_1+2\La c_2+c_3,\quad t\in [0, t_0).$$
Thus, we can apply the maximum principle
$$|A|(t)\leq \max_{L_0}|A|(0)+(8\La^3c_1+2\La c_2+c_3)t\leq \frac 32\La, \quad t\in [0, \frac {\La}{2(8\La^3c_1+2\La c_2+c_3)}].$$
Combining this with the definition of $t_0$ we have
$$t_0\geq\frac {\La}{2(8\La^3c_1+2\La c_2+c_3)}. $$

Now we estimate the mean curvature vector. In fact, for $t\in [0, t_0]$ the mean curvature satisfies
the inequality
$$\pd {}t|H|\leq \Delta |H|+|A|^2|H|+K_0|H|\leq \Delta|H|+(4\La^2+K_0)|H|,$$
which implies
$$|H|(t)\leq |H|(0)e^{(4\La^2+K_0)t}\leq 2\ee, \quad t\in [0, \min\{t_0, \frac {\log 2}{(4\La^2+K_0)}\}].$$
Thus, (\ref{lem:zero:eq1}) holds for $$T=\min\{\frac {\La}{2(8\La^3c_1+2\La c_2+c_3)}, \frac {\log 2}{(4\La^2+K_0)}\}.$$
\end{proof}

\medskip

For higher order estimates, K. Smoczyk proved in \cite{[S2]} that
all higher order derivatives of the second fundamental form are
bounded if the $C^0$ norm of $A$ is bounded for a short time
interval. However, the  bound of higher order derivatives will
depend on the derivatives of the second fundamental form of the
initial submanifold. In this paper we need more precise estimates as
in Ricci flow. The following result is taken from \cite{[CY]}, and
the readers are referred to \cite{[CY]} for details.

\begin{lem}(cf. Theorem 3.2 in \cite{[CY]})\label{lem:higer}
Assume that the Lagrangian mean curvature flow has a smooth solution for
 $t\in [0, t_0]$. If there is a
constant $\La$ such that
$$\max_{L_t}|A|^2\leq \La, \quad t\in [0, t_0],$$
then for any $k>0$ there exists a constant $C_k=C_k(n, \La, K_{k+1},
t_0)$ such that
$$\max_{L_t}|\Na^kA|^2\leq \frac {C_k}{t^k},\quad t\in (0, t_0],$$
where $K_k=\sum_{l=0}^k\;\max_M|\bar \Na^l\bar Rm|.$
\end{lem}

\begin{rem}\label{rem}In Lemma \ref{lem:higer} we can choose $t_0=T(n, \frac 12\La, K_1)$
where $T$ is given by Lemma \ref{lem:zero}, and  the constants $C_k$
depends only on $n, \La$ and $K_{k+1}.$ Thus, for any $t_1>t_0$ (no
matter how large $t_1$ is), as long as the flow satisfies
$\max_{L_t}|A|^2\leq \La$ when $ t\in [0, t_1],$ we have
$$\max_{L_t}|\Na^kA|^2\leq \frac {2^kC_k(n, \La, K_{k+1})}{t_0^k},\quad t\in \Big(\frac {t_0}2, t_1\Big],$$
Note that the right-hand side of the above inequality is independent
of $t_1.$ This property will be used many times in the proof of main
theorems.
\end{rem}
\section{Proof of Theorem \ref{main}}
In this section, we will prove Theorem \ref{main}. For any positive constants $\kappa, r, \La, \ee,$
we define the following subspace of Lagrangian submanifolds in $M$ by
$$\cA(\kappa, r, \La, \ee)=\Big\{L \;\Big|\;L \hbox{ is $\kappa$-noncollapsed on the scale $r$ with}\;
|A|(t)\leq \La, \;|H|(t)\leq \ee\Big\}.$$

\medskip

The following result shows that the flow will have good estimates
for a short time.
\begin{lem}\label{lem:initial}If the initial Lagrangian submanifold $L_0\in \cA(\kappa, r, \La, \ee)$,
then there exists $\tau=\tau(n, \La, K_1)>0$ such that $L_t\in \cA(\frac 12\kappa, r, 2\La, 2\ee)$ for $t\in [0, \tau].$
\end{lem}
\begin{proof}This result follows directly from Lemma \ref{lem:vol} and Lemma \ref{lem:zero}.
\end{proof}

\medskip

The following lemma is a crucial step in the whole argument of the
proof. It shows that if the flow has a rough bound for a finite time
interval, then we can choose some constant sufficiently small such
that the mean curvature will decay exponentially and the flow has
uniform bounds which are independent of the length of this time
interval.

\begin{lem}\label{lem:main}For any $\kappa_0, r_0, \La_0, V_0, T>0$ there exists $\ee_0=\ee_0(\kappa_0, r_0, \La_0, n, K_5, V_0)>0$
such that if the solution $L_t(t\in [0, T])$ of the Lagrangian mean curvature flow satisfies
\begin{enumerate}
  \item  $L_0\in \cA(\kappa_0, r_0, \La_0, \ee_0)$ and $\vol(L_0)\leq V_0$,

  \item $L_t\in \cA(\frac 13\kappa_0, r_0, 6\La_0, 2\ee_0^{\frac 1{n+2}})(t\in [0, T])$,
\end{enumerate}
Then we have the following properties
\begin{enumerate}
  \item[(a)] The mean curvature vector satisfies
  $$\max_{L_t}|H|(t)\leq \ee_0^{\frac 1{n+2}}e^{\frac {\bar R}{2n(n+2)}t},\quad t\in [\tau, T].$$
  \item[(b)] The second fundamental form $$\max_{L_t}|A|\leq 3\La_0,\quad t\in [0, T].$$
  \item[(c)] $L_t$ is $\frac 23\kappa_0$-noncollapsed on the scale $r_0$ for $t\in [0, T].$
\end{enumerate}
Thus, the solution $L_t\in\cA(\frac 23\kappa_0, r_0, 3\La_0, \ee_0^{\frac 1{n+2}})$ for $t\in [0, T],$ and
by Lemma \ref{lem:initial} we can
extend the solution to $[0, T+\dd]$ such that $L_t\in \cA(\frac 13\kappa_0, r_0, 6\La_0, 2\ee_0^{\frac 1{n+2}})(t\in [0, T+\dd])$
for some $\dd=\dd(n, \La_0, K_1)>0.$
\end{lem}
\begin{proof}$(a)$. For any $\La_0>0,$ we can choose $\ee_0$ small enough such that $12\La_0 \ee_0^{\frac 1{n+2}}<-\frac {\bar R}{4n}.$
Thus, by Lemma \ref{Lem:H} the mean curvature vector satisfies
\beq\int_{L_t}\; |H|^2d\mu_t\leq e^{\frac {\bar R}{2n}t}\int_{L_0}\;
 |H|^2d\mu_0\leq V_0\ee_0^2e^{\frac {\bar R}{2n}t},\quad t\in [0, T].\label{lem:main:eq1}\eeq
Note that $L_t\in \cA(\frac 13\kappa_0, r_0, 6\La_0, 2\ee_0^{\frac
1{n+2}})$ for $t\in [0, T]$, by Lemma \ref{lem:higer} and Remark
\ref{rem} there is a constant $C_1=C_1(n, \La_0, K_2)$ such that
\beq |\Na A|(t)\leq C_1(n, \La_0, K_2, \tau),\quad t\in [\tau,
T].\label{lem:main:eq2} \eeq Here we can choose $\tau=\tau(n, \La_0,
K_1)$ in Lemma \ref{lem:initial}. Thus, by Lemma \ref{lem:f} and
(\ref{lem:main:eq1})(\ref{lem:main:eq2}) we have \beq |H|(t)\leq
\Big(\sqrt{\frac 3{\kappa_0}}+C_1\Big)V_0^{\frac 1{n+2}}\ee_0^{\frac
2{n+2}}e^{\frac {\bar R}{2n(n+2)}t},\quad t\in [\tau, T]. \eeq where
we have used the fact that $L_t$ is $\frac {\kappa}3$-noncollapsed
on the scale $r_0$ and $V_0\ee_0^2\leq r_0^{n+2}$ if $\ee_0$ is
small enough. Thus, if $\ee_0$ is  small such that $\Big(\sqrt{\frac
3{\kappa_0}}+C_1\Big)V_0^{\frac 1{n+2}}\ee_0^{\frac 1{n+2}}\leq 1,$
then we have \beq |H|(t)\leq \ee_0^{\frac 1{n+2}}e^{\frac {\bar
R}{2n(n+2)}t},\quad t\in [\tau, T].\n \eeq

$(b)$. By Lemma \ref{lem:higer} and Remark \ref{rem} there exist
some constants $C_k=C_k(n, \La_0, K_{k+1})$ such that \beq |\Na^k
A|(t)\leq C_{k}(n, \La_0, K_{k+1}, \tau),\quad t\in [\tau,
T].\label{lem:main:eq3} \eeq By Lemma \ref{lem:higer} and Property
$(a)$, we have \beq \int_{L_t}|\Na^2 H|^2d\mu_t\leq
\int_{L_t}|H||\Na^4H|d\mu_t\leq V_0C_4\ee_0^{\frac 1{n+2}} e^{\frac
{\bar R}{2n(n+2)}t} ,\quad t\in [\tau, T], \eeq where we used the
fact that $\vol(L_t)$ is decreasing along the flow since
$$\pd {}t\vol(L_t)=-\int_{L_t}|H|^2d\mu_t\leq 0.$$
Thus, by Lemma \ref{lem:f} we have
\beq |\Na^2H|\leq \Big(\sqrt{\frac 3{\kappa_0}}+C_3\Big)C_4^{\frac 1{n+2}}V_0^{\frac 1{n+2}}
\ee_0^{\frac 1{(n+2)^2}}e^{\frac {\bar R}{2n(n+2)^2}t},\quad t\in [\tau, T].\label{lem:main:eq4}\eeq
Recall that by Lemma \ref{lem:LH} $|A|$ satisfies the inequality
\beq
\pd {}t|A|\leq |\Na^2H|+c(n)|A|^2|H|+|\bar Rm||H|.\label{lem:main:eq5}
\eeq
Thus, by Lemma \ref{lem:zero}, (\ref{lem:main:eq4})(\ref{lem:main:eq5})  and $(a)$ we have
\beqn
|A|(t)&\leq&|A|({\tau})+\int_{\tau}^t |\Na^2H|+(K_0+|A|^2)|H|\n\\
&\leq &2\La_0+\Big(\sqrt{\frac 3{\kappa_0}}+C_3\Big)C_4^{\frac 1{n+2}}V_0^{\frac 1{n+2}}
\ee_0^{\frac 1{(n+2)^2}} \frac {2n(n+2)^2}{|\bar R|}\n\\
&&+(K_0+36\La_0^2)\ee_0^{\frac 1{n+2}}\frac {2n(n+2)}{|\bar R|}\n\\
&\leq&3\La_0,
\eeqn
if we choose $\ee_0$ sufficiently small.

(3). By (\ref{lem:vol:eq1}), Lemma \ref{lem:initial} Property $(a)(b)$ we have
\beqs E(t)&\leq& \int_{0}^{\tau}\;\max_{L}(|A||H|+|H|^2)\;ds+\int_{\tau}^t\;\max_{L}(|A||H|+|H|^2)\;ds\\
&\leq&4\La_0\ee_0\tau+4\ee_0^2\tau+3\La_0\ee_0^{\frac 1{n+2}} \frac {2n(n+2)}{\bar R}+\ee_0^{\frac 2{n+2}}\frac {n(n+2)}{\bar R}\\
&\leq &\frac 1{n+1}\log \frac 32, \quad t\in [0, T],
\eeqs
where $\ee_0$ is small enough. Thus, by Lemma \ref{lem:vol} \,$ L_t$ is $\frac 23\kappa_0$-noncollapsed on the scale $r_0$ for $t\in [0, T].$

\end{proof}

Now we can prove the following stability result, which needs the
noncollapsing condition of the initial submanifold. This condition
can be removed by the comparison theorem in Theorem \ref{main}.

\begin{theo}\label{main01}
Let $(M, \bar g)$ be a complete K\"ahler-Einstein manifold satisfying (\ref{eq:00}) with scalar
curvature $\bar R<0$, and $L$ be a compact Lagrangian submanifold
smoothly immersed in $M$. For any $\kappa_0, r_0, V_0, \La_0>0$,
there exists $\ee_0=\ee_0(\kappa_0, r_0, V_0, \La_0, \bar R, K_5)$
such that if $L$ is $\kappa_0$-noncollapsed on the scale $r_0$ and
satisfies
$$\vol(L)\leq V_0, \quad |A|\leq \La_0,\quad  |H|\leq \ee_0,$$
 then the Lagrangian mean  curvature flow with the initial data $L$ will converge exponentially fast to a minimal
Lagrangian submanifold in $M.$
\end{theo}

\begin{proof}. Suppose that $L_0\in \cA(\kappa_0, r_0, \La_0, \ee_0)$
for any positive constants $\kappa_0, r_0, \La_0$ and small $\ee_0$ which will be chosen later. Define
$$t_0=\sup\Big\{t>0\;\Big|\; L_s\in \cA(\frac 13\kappa_0, r_0, 6\La_0, 2\ee_0^{\frac 1{n+2}} ),\;\;s\in [0, t)\Big\}.$$
Suppose that $t_0<+\infty.$ By Lemma \ref{lem:main}, there exists $\ee_0=\ee_0(\kappa_0, r_0, \La_0, n, K_5, V_0)$
such that $L_t\in\cA(\frac 23\kappa_0, r_0, 3\La_0, \ee_0^{\frac 1{n+2}})$ for all $t\in [0, t_0).$
Moreover, by Lemma \ref{lem:main} again the solution $L_t$ can be extended to $[0, t_0+\dd]$ such that
$L_t\in \cA(\frac 13\kappa_0, r_0, 6\La_0, 2\ee_0^{\frac 1{n+2}})$, which contradicts the definition of $t_0$.
Thus, $t_0=+\infty$ and $$L_t\in \cA(\frac 13\kappa_0, r_0, 6\La_0, 2\ee_0^{\frac 1{n+2}} ),\quad t\in [0, \infty).$$
By Lemma \ref{lem:main} the mean curvature vector will decay exponentially to zero and the flow will
converge to a smooth minimal Lagrangian submanifold. The theorem is proved.

\end{proof}

We can finish the proof of Theorem \ref{main} as follows:

\begin{proof}[Proof of Theorem \ref{main}] We will show that under the assumption of Theorem \ref{main},
the flow $L_t$ will satisfies all the conditions in Theorem \ref{main01} after a short time. Suppose that the initial Lagrangian submanifold $L$ satisfies
(\ref{main01:eq1}), by Lemma \ref{lem:zero} there exists $T=T(n, \La, K_1)$ such that
\beq |A|(t)\leq 2\La,\quad t\in [0, T].\label{main01:eq3}\eeq

\medskip

We claim that there exists $t_0=t_0(n, \La, K_1)<T$ such that
the $L^2$ norm of the mean curvature vector satisfies\beq
\int_{L_t}\;|H|^2\,d\mu_t\leq 2\ee_0,\quad t\in [0, t_0].\label{main01:eq2}
\eeq
In fact, by (\ref{lem:H:eq2}) in Lemma \ref{Lem:H} we have
\beqn
\pd
{}t\int_{L_t}\;|H|^2d\mu_t&\leq&\int_{L_t}\;2\bar R_{ij}H^iH^j-2H^iH^jH^kh_{jk}^i-|H|^4\n\\
&\leq&\Big(\frac {\bar R}n+4\La^2\Big)\int_{L_t}\;|H|^2d\mu_t,
\eeqn
where we used (\ref{main01:eq3}) and the inequality
$$2H^iH^jH^kh_{jk}^i\leq 4\La^2|H|^2+|H|^4.$$
Thus, we have
\beq
\int_{L_t}\;|H|^2d\mu_t\leq e^{4\La^2 t}\int_{L_0}\;|H|^2d\mu_0\leq \ee_0e^{4\La^2 t},\quad t\in [0, T].
\eeq
which proves (\ref{main01:eq2}) if we choose $t_0$ sufficiently small.

\medskip

Now we prove that there exist $\kappa_0, r_0>0$ such that $L_t$ is $\kappa_0$-noncollapsed on the
scale $r_0$ for $t\in [\frac 1{2}t_0, t_0].$ In fact,
by Proposition 2.2 in \cite{[CH]} or Theorem 2.1 in \cite{[CY]} the injectivity radius of $L$
is bounded from below
\beq inj(L_t)\geq \iota,\quad t\in [\frac 12t_0, t_0]\label{main01:eq4}\eeq for some constant
 $\iota=\iota(n, \La, K_0, \iota_0)$. By (\ref{main01:eq3}) and by Gauss equation the intrinsic curvature of $L_t$
is uniformly bounded
\beq |Rm|\leq C(K_0, \La), \quad t\in [\frac 12t_0, t_0].\label{main01:eq5}\eeq
By (\ref{main01:eq4})(\ref{main01:eq5}) together with the volume comparison theorem, there exist $\kappa_0=\kappa_0(n, \iota_0, K_0, \La)$ and
$r_0=r_0(n, \iota_0, K_0, \La)$ such that $L_t$ is $\kappa_0$-noncollapsed on the scale $r_0$ for all $t\in [\frac 12t_0, t_0].$

\medskip

By (\ref{main01:eq3}) and Lemma \ref{lem:higer}  the derivative of the second fundamental form is uniformly bounded
$$|\Na A|\leq C_1(n, \La, K_2),\quad t\in [\frac 12 t_0, t_0].$$
Now we can apply Lemma \ref{lem:f}  to show that
$$|H|(t)\leq \Big(\frac 1{\sqrt{\kappa_0}}+2C_1\Big)(2\ee_0)^{\frac 1{n+2}}\quad t\in [\frac 12 t_0, t_0].$$

In summary, all the conditions in Theorem \ref{main01} are satisfied for $L_t(t\in [\frac 12t_0, t_0])$, and thus Theorem \ref{main} is proved.

\end{proof}
\section{Proof of Theorem \ref{main2}}
In this section, we will prove Theorem \ref{main2}. The idea of the
proof is similar to that of Theorem \ref{main} but more involved
since we need to consider the evolution of the first eigenvalue.

   For any positive constants $\dd, \kappa, r, \La, \ee,$
we define the following subspace of Lagrangian submanifolds in $M$ by
\beqs \cB(\kappa, r, \dd, \La, \ee)=\Big\{\;L &\Big|\;\;L \hbox{ is $\kappa$-noncollapsed on the scale $r$ with}\\
&\; \la_1\geq \frac {\bar R}{2n}+\dd,\;\;|A|(t)\leq \La, \;\;|H|(t)\leq \ee\;\Big\}.\eeqs

\begin{lem}\label{lem:initial2}If the initial Lagrangian submanifold $L_0\in \cB(\kappa, r, \dd, \La, \ee)$,
then there exists $\tau=\tau(n, \La, \dd, K_2, \bar R)>0$ such that $L_t\in \cB(\frac 12\kappa, r, \frac {2\dd}3, 2\La, 2\ee)$ for $t\in [0, \tau].$
\end{lem}
\begin{proof}This result follows directly from Lemma \ref{lem:vol} and Lemma \ref{lem:zero}.
\end{proof}

\begin{lem}\label{lem:main2}For any $\kappa_0, r_0, \dd_0, \La_0, V_0, T>0$ there exists $\ee_0=\ee_0(\kappa_0, r_0, \dd_0, \bar R, \La_0, n, K_5, V_0)>0$
such that if the solution $L_t(t\in [0, T])$ of the Lagrangian mean curvature flow satisfies
\begin{enumerate}
  \item  $L_0\in \cB(\kappa_0, r_0, \dd_0, \La_0, \ee_0)$ and $\vol(L_0)\leq V_0$,

  \item $L_t\in \cB(\frac 13\kappa_0, r_0, \frac 13\dd_0,  6\La_0, 2\ee_0^{\frac 1{n+2}})(t\in [0, T])$,
\end{enumerate}
Then we have the following properties
\begin{enumerate}
  \item[(a)] The mean curvature vector satisfies
  $$\max_{L_t}|H|(t)\leq \ee_0^{\frac 1{n+2}}e^{-\frac {\dd_0}{2(n+2)}t},\quad t\in [\tau, T].$$
  \item[(b)] The second fundamental form $$\max_{L_t}|A|\leq 3\La_0,\quad t\in [0, T].$$
  \item[(c)] $L_t$ is $\frac 23\kappa_0$-noncollapsed on the scale $r_0$ for $t\in [0, T].$
  \item[(d)] The first eigenvalue
  $$\la_1(t)\geq \frac {\bar R}{2n}+\frac {\dd_0}{2},\quad t\in [0, T].$$
\end{enumerate}
Thus, the solution $L_t\in\cB(\frac 23\kappa_0, r_0, \frac {\dd_0}{2}, 3\La_0, \ee_0^{\frac 1{n+2}})$ for $t\in [0, T],$ and
by Lemma \ref{lem:initial2} we can
extend the solution to $[0, T+\dd]$ such that $L_t\in \cB(\frac 13\kappa_0, r_0, \frac 13\dd_0, 6\La_0, 2\ee_0^{\frac 1{n+2}})(t\in [0, T+\dd])$
for some $\dd=\dd(n, \La_0, K_1)>0.$
\end{lem}

\begin{proof}$(a)$. By assumption 2, the first eigenvalue satisfies $\la_1(t)\geq \frac {\bar R}{2n}+\frac {\dd_0}3$.
Thus, we have
\beq
\la_1-\frac {\bar R}{2n}-12\La_0 \ee^{\frac 1{n+2}}\geq \frac {\dd_0}{4},\quad t\in [0, T] \n
\eeq when $\ee_0$ is small enough.
Thus, by (\ref{lem:H:eq3}) in Lemma \ref{Lem:H} the mean curvature vector satisfies
\beq\int_{L_t}\; |H|^2d\mu_t\leq e^{-\frac {\dd_0}{2}t}\int_{L_0}\;
 |H|^2d\mu_0\leq V_0\ee_0^2e^{-\frac {\dd_0}{2}t},\quad t\in [0, T].\label{lem:main2:eq1}\eeq
Note that $L_t\in \cB(\frac 13\kappa_0, r_0, \frac 13\dd_0,  6\La_0,
2\ee_0^{\frac 1{n+2}})$ for $t\in [0, T]$, by Lemma \ref{lem:higer}
and Remark \ref{rem} there is a constant $C_1=C_1(n, \La_0, K_2)$
such that \beq |\Na A|(t)\leq C_1(n, \La_0, K_2, \tau),\quad t\in
[\tau, T].\label{lem:main2:eq2} \eeq Here we can choose
$\tau=\tau(n, \La_0, \dd_0, K_2, \bar R)$ in Lemma
\ref{lem:initial2}. Thus, by Lemma \ref{lem:f} and
(\ref{lem:main2:eq1})(\ref{lem:main2:eq2}) we have \beq |H|(t)\leq
\Big(\sqrt{\frac 3{\kappa_0}}+C_1\Big)V_0^{\frac 1{n+2}}\ee_0^{\frac
2{n+2}}e^{-\frac {\dd_0}{2(n+2)}t},\quad t\in [\tau, T]. \eeq where
we have used the fact that $L_t$ is $\frac {\kappa}3$-noncollapsed
on the scale $r_0$ and $V_0\ee_0^2\leq r_0^{n+2}$ if $\ee_0$ is
small enough. Thus, if $\ee_0$ is  small such that $\Big(\sqrt{\frac
3{\kappa_0}}+C_1\Big)V_0^{\frac 1{n+2}}\ee_0^{\frac 1{n+2}}\leq 1,$
then we have \beq |H|(t)\leq \ee_0^{\frac 1{n+2}}e^{-\frac
{\dd_0}{2(n+2)}t},\quad t\in [\tau, T].\n \eeq

$(b)$. By Lemma \ref{lem:higer} and Remark \ref{rem} there exist
some constants $C_k=C_k(n, \La_0, K_{k+1})$ such that \beq |\Na^k
A|(t)\leq C_{k}(n, \La_0, K_{k+1}, \tau),\quad t\in [\tau,
T].\label{lem:main2:eq3} \eeq By Lemma \ref{lem:higer} and Property
$(a)$, we have \beq \int_{L_t}|\Na^2 H|^2d\mu_t\leq
\int_{L_t}|H||\Na^4H|d\mu_t\leq V_0C_4\ee_0^{\frac 1{n+2}} e^{-\frac
{\dd_0}{2(n+2)}t} ,\quad t\in [\tau, T].\n \eeq Thus, by Lemma
\ref{lem:f} we have \beq |\Na^2H|\leq \Big(\sqrt{\frac
3{\kappa_0}}+C_3\Big)C_4^{\frac 1{n+2}}V_0^{\frac 1{n+2}}
\ee_0^{\frac 1{(n+2)^2}}e^{-\frac {\dd_0}{2(n+2)^2}t},\quad t\in
[\tau, T].\label{lem:main2:eq4}\eeq Recall that by Lemma
\ref{lem:LH} $|A|$ satisfies the inequality \beq \pd {}t|A|\leq
|\Na^2H|+c(n)|A|^2|H|+|\bar Rm||H|.\label{lem:main2:eq5} \eeq Thus,
by Lemma \ref{lem:initial2},
(\ref{lem:main2:eq4})(\ref{lem:main2:eq5})  and $(a)$ we have \beqn
|A|(t)&\leq&|A|({\tau})+\int_{\tau}^t |\Na^2H|+(K_0+|A|^2)|H|\n\\
&\leq &2\La_0+\Big(\sqrt{\frac 3{\kappa_0}}+C_3\Big)C_4^{\frac 1{n+2}}V_0^{\frac 1{n+2}}
\ee_0^{\frac 1{(n+2)^2}} \frac {2(n+2)^2}{\dd_0}\n\\
&&+(K_0+36\La_0^2)\ee_0^{\frac 1{n+2}}\frac {2(n+2)}{\dd_0}\n\\
&\leq&3\La_0,
\eeqn
if we choose $\ee$ sufficiently small.

$(c)$. By (\ref{lem:vol:eq1}), Lemma \ref{lem:initial2} Property $(a)(b)$ we have
\beqs E(t)&\leq& \int_{0}^{\tau}\;\max_{L}(|A||H|+|H|^2)\;ds+\int_{\tau}^t\;\max_{L}(|A||H|+|H|^2)\;ds\\
&\leq&4\La_0\ee_0\tau+4\ee_0^2\tau+3\La_0\ee_0^{\frac 1{n+2}} \frac {2(n+2)}{\dd_0}+\ee_0^{\frac 2{n+2}}\frac {(n+2)}{\dd_0}\\
&\leq &\frac 1{n+1}\log \frac 32, \quad t\in [0, T],
\eeqs
where $\ee_0$ is small enough. Thus, by Lemma \ref{lem:vol}\,$ L_t$ is $\frac 23\kappa_0$-noncollapsed on the scale $r_0$ for $t\in [0, T].$

$(d)$. By (\ref{lem:main2:eq3}) and Property $(a)$ we have
\beq
\int_{L_t}\;|\Na H|^2\leq \int_{L_t}\;|H||\Na^2 H|\leq V_0 C_2\ee_0^{\frac 1{n+2}}e^{-\frac {\dd_0}{2(n+2)}t},\quad t\in [\tau, T]
\n\eeq
Thus, by Lemma \ref{lem:f} we have
\beqn |\Na H|&\leq& \Big(\sqrt{\frac 3{\kappa_0}}+C_2\Big)C_2^{\frac 1{n+2}}V_0^{\frac 1{n+2}}
\ee_0^{\frac 1{(n+2)^2}}e^{-\frac {\dd_0}{2(n+2)^2}t}\n\\
&\leq &\ee_0^{\frac 1{2(n+2)^2}}e^{-\frac {\dd_0}{2(n+2)^2}t},\quad t\in [\tau, T],\label{lem:main2:eq6}\eeqn
if $\ee_0$ is sufficiently small. Thus, we have
\beq
|H|+|\Na H|\leq 2\ee_0^{\frac 1{2(n+2)^2}}e^{-\frac {\dd_0}{2(n+2)^2}t}
\eeq
Note that by Lemma \ref{lem:initial2} the first eigenvalue
$$\la_1(t)\geq \frac {\bar R}{2n}+\frac {2\dd_0}{3},\quad t\in [0, \tau].$$
Thus, by (\ref{lem:eig:eq5}) in Lemma \ref{lem:eig} we have
\beqn
\sqrt{\la_1(t)}&\geq& \sqrt{\la_1(\tau)}e^{-\frac {(n+2)^2}{\dd_0}\cdot \Big( 24\La_0 \ee_0^{\frac 1{2(n+2)^2}}+4\ee_0^{\frac 1{(n+2)^2}}\Big)}
-\frac {2(n+2)^2}{\dd_0}(K_0+6\La_0) \cdot 2\ee_0^{\frac 1{2(n+2)^2}}\n\\
&\geq&\sqrt{\frac {\bar R}{2n}+\frac {2\dd_0}{3}}e^{-\frac {(n+2)^2}{\dd_0}\cdot \Big( 24\La_0 \ee_0^{\frac 1{2(n+2)^2}}+4\ee_0^{\frac 1{(n+2)^2}}\Big)}
-\frac {2(n+2)^2}{\dd_0}(K_0+6\La_0) \cdot 2\ee_0^{\frac 1{2(n+2)^2}}.\n
\eeqn Thus, if $\ee_0$ is small enough we have
$$\la_1(t)\geq \frac {\bar R}{2n}+\frac {\dd_0}{2},\quad t\in [0, T].$$
The lemma is proved.

\end{proof}

As in the proof of Theorem \ref{main}, we can see that Theorem \ref{main2} follows directly from
Lemma \ref{lem:eig} and the result:

\begin{theo}\label{main20}
Let $(M, \bar g)$ be a complete K\"ahler-Einstein manifold satisfying (\ref{eq:00}) with scalar
curvature $\bar R\geq 0$, and $L$ be a compact Lagrangian
submanifold smoothly immersed in $M$. For any $\kappa_0, r_0, V_0,
\La_0, \dd_0>0$, there exists $\ee_0=\ee_0(\kappa_0, r_0, V_0,
\La_0, \bar R, \dd_0,  K_5)>0$ such that if
\begin{enumerate}
\item  the mean curvature form of $L$ is exact,
  \item $L$ is $\kappa_0$-noncollapsed on the scale $r_0$,
  \item $L$ satisfies
  $$\la_1\geq \frac {\bar R}{2n}+\dd_0,\quad\vol(V)\leq V_0, \quad |A|\leq \La_0,\quad  |H|\leq \ee_0,$$
  where $\la_1$ is the first eigenvalue of the Laplacian operator with respect to the induced metric on $L,$
\end{enumerate}
 then the Lagrangian mean  curvature flow with the initial data $L$ will converge exponentially fast to a minimal
Lagrangian submanifold in $M.$

\end{theo}

\begin{proof}. Suppose that $L_0\in \cB(\kappa_0, r_0, \dd_0, \La_0, \ee_0)$
for any positive constants $\kappa_0, r_0, \dd_0, \La_0$ and small $\ee_0$ which will be chosen later. Define
$$t_0=\sup\Big\{t>0\;\Big|\; L_s\in \cB(\frac 13\kappa_0, r_0, \frac 13\dd_0,  6\La_0, 2\ee_0^{\frac 1{n+2}}),\;\;s\in [0, t)\Big\}.$$
Suppose that $t_0<+\infty.$ By Lemma \ref{lem:main2}, there exists $\ee_0=\ee_0(\kappa_0, r_0, \dd_0, \bar R, \La_0, n, K_5, V_0)>0$
such that $L_t\in\cB(\frac 23\kappa_0, r_0, \frac {\dd_0}2, 3\La_0, \ee_0^{\frac 1{n+2}})$ for all $t\in [0, t_0).$
Moreover, by Lemma \ref{lem:main2} again the solution $L_t$ can be extended to $[0, t_0+\dd]$ such that
$L_t\in \cB(\frac 13\kappa_0, r_0, \frac 13\dd_0, 6\La_0, 2\ee_0^{\frac 1{n+2}})$, which contradicts the definition of $t_0$.
Thus, $t_0=+\infty$ and $$L_t\in \cB(\frac 13\kappa_0, r_0, \frac 13\dd_0, 6\La_0, 2\ee_0^{\frac 1{n+2}} ),\quad t\in [0, \infty).$$
By Lemma \ref{lem:main2} the mean curvature vector will decay exponentially to zero and the flow will
converge to a smooth minimal Lagrangian submanifold. The theorem is proved.

\end{proof}

\section{Proof of Theorem \ref{main3}}
In this section, we will introduce some definitions related to the
deformation of a Lagrangian submanifold, and prove the exponential
decay of the mean curvature vector under the Lagrangian mean
curvature flow with some special initial data. The idea of the
argument is very similar to K\"ahler-Ricci flow in a
K\"ahler-Einstein manifold with nonzero holomorphic vector fields
(cf. \cite{[CL]}\cite{[CLW]}).

\subsection{Deformation of minimal Lagrangian submanifolds}

Let $(M, \bar g)$ be a K\"ahler-Einstein manifold. First we give some definitions(cf.  \cite{[Oh]}):

\begin{defi}(1). Let $L\subset M$ be a Lagrangian submanifold and $X$ be
a vector field along $L$. $X$ is called a Lagrangian(resp. hamiltonian) variation if
its associated one form $$\al_X=i_X\bar \oo$$ is closed(resp. exact), where $\bar \oo$ is the
K\"ahler form of the metric $\bar g$ on $M.$

\medskip

(2). A smooth family $\phi_s$ of immersions of $L$ into $M$ is called a Lagrangian (resp. hamiltonian)
deformation if its derivative
$$X=\pd {\phi_s(L)}s$$
 is Lagrangian (resp. hamiltonian) for each $s.$

\end{defi}

In the following, we assume that $\phi_0: L\ri M$ is a smooth minimal Lagrangian submanifold
into a K\"ahler-Einstein manifold $(M, \bar g)$, and $X=J\Na f_0$ is a hamiltonian variation of $L_0=\phi_0(L).$
We remind that the notation $L_0$ has different meaning in previous sections, and the readers should
not confuse it.

We can extend the vector $X$ to a neighborhood of $L_0$ in $M$ such that it is still hamiltonian. Let
$\phi_s: L\ri M(s\in (-\eta, \eta))$ be a family of hamiltonian deformations generated by $X$ and we write
$L_s=\phi_s(L_0).$  For the hamiltonian deformation $L_s$, we have the following result:

\medskip

\begin{lem}\label{lem:defo}Let $f_s$ be a smooth function such that
\beq \pd {L_s}{s}=J\Na f_s,\label{lem:defo:eq2}\eeq
then the Lagrangian angle $\theta_s$ of $L_s$ satisfies
\beq
\pd {\theta_s}s=-\Delta_s f_s-\frac {\bar R}{2n}f_s.\label{lem:defo:eq1}
\eeq

\end{lem}
\begin{proof}Let $\{e_1,\cdots, e_n\}$ be a normal coordinate frame on $L_s$ with $e_i=\p_i \phi_s.$ Since
$L_s$ is Lagrangian for each $s$, the vectors $Je_1, \cdots, Je_n $ are orthogonal to $L.$ The induced metric
on $L_s$ is $g_{ij}=\bar g(e_i, e_j).$ By (\ref{lem:defo:eq2}) we have
\beq
\pd {g_{ij}}s=2\Na^k f_s h^k_{ij}. \label{lem:defo:eq4}
\eeq
By the same calculation as in Lemma \ref{lem:LH}, the second fundamental form satisfies
\beq
\pd {}th_{ij}^{k}
=-\Na_i\Na_j\Na^kf_s+\Na^m f_s h_{jl}^{m}h_{il}^{k}+\Na^lf_sh^l_{mk}
h_{ij}^{m}-\Na^mf_s\bar R_{\un k j\un m
i}\n
\eeq
and the mean curvature vector
\beq
 \pd {}tH^{i}=-\Delta\Na^if_s-\Na^j f_sh_{lm}^{j}h_{ml}^{i}-\Na^jf_s\bar R_{\un i l \un j l}+\Na^jf_sH^kh^i_{kj}.\label{lem:defo:eq3}
\eeq
Since $L_s$ is a hamiltonian deformation, we can write $H^i=\Na^i\theta_s$. Thus, by the same calculation in the proof
of Lemma \ref{lem:LH} and (\ref{lem:defo:eq3}) we have
\beq
\pd {\theta_s}s=-\Delta_s f_s-\frac {\bar R}{2n}f_s.\n
\eeq
The lemma is proved.

\end{proof}

\medskip

Recall that a minimal Lagrangian submanifold is called hamiltonian stable (resp. strictly stable), if for any hamiltonian
variation $X$ the second variation along $X$ of the volume functional is nonnegative (resp. positive). When $M$ is a
K\"ahler-Einstein manifold with scalar curvature $\bar R,$ Oh \cite{[Oh]} proved that a compact minimal
Lagrangian submanifold $L$ is hamiltonian stable if and only if the first eigenvalue of the Laplacian operator
on $L$ has $\la_1\geq \frac {\bar R}{2n}.$

\medskip
Now we introduce the definition:
\begin{defi} \label{defi:essential}
 A nonzero vector field $X$ is called an essential hamiltonian variation of a Lagrangian submanifold $L_0$ , if
$X$ can be written as $X=J\Na f$ where $f\notin E_{\la_1}$, where $E_{\la_1}$ is the first eigenspace of the Laplacian operator $\Delta$ on $L$.
\end{defi}

For an essential hamiltonian vector $X$ on a minimal Lagrangian submanifold $L_0$, we can show that $L_0$
is strictly hamiltonian stable along the variation $X$ in the following sense:

\begin{lem}\label{lem:positive}Let $\phi_s: L\ri M$ be a hamiltonian deformation of a minimal Lagrangian submanifold
$L_0=\phi_0(L)$ with $\la_1=\frac {\bar R}{2n}$ and $X=\pd {\varphi}{s}|_{s=0}$.
Then $X$ is  an essential hamiltonian variation on $L_0$ if and only if
\beq
\frac {d^2}{ds^2}\vol(L_s)\Big|_{s=0}>0.\n
\eeq

\end{lem}

\begin{proof}Let $X_s=\pd {\phi_s}{s}$. Since $L_s$ is a hamiltonian deformation, we can find smooth functions
$f_s$ such that $X_s=\Na f_s.$ Now we calculate
\beqn
\frac d{ds}\vol(L_s)&=&\int_{L_s}\;\frac 12 g^{ij}\pd {g_{ij}}s\,d\mu_s\n\\
&=&\int_{L_s}\;-\theta_s\Delta_s f_s \,d\mu_s,\n
\eeqn where we used (\ref{lem:defo:eq4}). Thus, the second variation of the volume is
\beq
\frac {d^2}{ds^2}\vol(L_s)\Big|_{s=0}=\int_{L_0}\;\Big(\Delta_0 f_0+\frac {\bar R}{2n}f_0\Big)\Delta_0 f_0\, d\mu_0\label{lem:defo:eq5}
\eeq
where we used the equality (\ref{lem:defo:eq1}) and the fact that $L_0$ is minimal.
By the eigenvalue decomposition, we can assume that
$$f_0=\sum_{i=1}^{\infty}\;a_i\eta_i,$$
where the functions $\eta_i$ satisfies
$$-\Delta_0 \eta_i=\la_i \eta_i,\quad \int_{L_0}\;\eta_i^2\,d\mu_0=1$$ for the eigenvalues $\la_1<\la_2<\cdots.$
Thus, (\ref{lem:defo:eq5}) can be written as
\beq
\frac {d^2}{ds^2}\vol(L_s)\Big|_{s=0}=\sum_{i=1}^{\infty}\;a_i^2 \la_i(\la_i-\frac {\bar R}{2n})\geq 0,
\eeq
where the equality holds if and only if $a_i=0$ for all $i\geq 2,$ which says $f_0\in E_{\la_1}.$ The lemma
is proved.
\end{proof}

\subsection{Exponential decay of the mean curvature vector}

\medskip
To proceed further, we need the following compactness result for mean curvature flow.

\begin{prop}(cf. \cite{[CH]})\label{compactness}
Let $\phi_k(t): L\subset M$ be a sequence of mean curvature flow from a compact submanifold $L$ to
a compact Riemannian manifold $M$ with uniformly bounded second fundamental forms
$$|A_k|(t)\leq C,\quad \forall \,t\in [0, T].$$
Then there exists a sequence of $\phi_k(t)$ which converges to a mean curvature flow $\phi_{\infty}(t)(t\in (0, T)))$
and $L_{\infty}=\phi_{\infty}(t)(L)$ is a smooth Riemannian manifold.
\end{prop}
\begin{proof}The proposition is proved by Chen-He in \cite{[CH]} for the case when the ambient manifold is the
Euclidean space. For a general compact ambient manifold $M$, we can embed $M$ isometrically into $\RR^N$
for some large $N$ and the corresponding mean curvature of the submanifold $\phi_k(L)(t)$ in $\RR^N$ is still
uniformly bounded. Thus, we can apply Chen-He's theorem and the proposition is proved.

\end{proof}

We denote by $L_{s, t}=\phi_{s, t}(L_0)(t\in [0, T])$ the Lagrangian mean curvature flow
 with the initial data $L_s$. Since $L_s$ is a hamiltonian deformation of $L_0$, the mean curvature
 form of $L_s$ is exact for each $s$. Thus, the mean curvature form of $L_{s, t}$ is also exact, and we
 denote the Lagrangian angle by $\theta_{s, t}$. Suppose that the deformation $L_s$ is sufficiently close
 to $L_0$ in the following sense
 \beq
\|\phi_s-\phi_0\|_{C^3}\leq \ee_0 \label{main3:eq1}
 \eeq
for small $\ee_0$ which will be determined later.  The next lemma shows that  $\theta_{s, t}$ satisfies certain inequality if $L_s$ is sufficiently close
to $L_0:$

\begin{lem}\label{lem:main31}Let $X=J\Na f_0$ be an essential hamiltonian variation of $L_0$,
where $L_0$ is a minimal Lagrangian submanifold with the first
eigenvalue $\la_1=\frac {\bar R}{2n}$.  For any $\La>0$, there
exists $\ee_0=\ee_0(L_0, X, M)>0$ and $\dd_0>0$ such that if $L_{s,
t}$ satisfies \beq |A_s|(t)\leq\La, \quad |H_s|(t)\leq \ee_0,\quad
\forall\, t\in [0, T] \eeq then the Lagrangian angle $\theta_{s, t}$
of $L_{s, t}$ satisfies \beq \int_{L_{s, t}}\;|\Delta \theta_{s,
t}|^2\geq(\frac {\bar R}{2n}+\dd_0)\int_{L_{s, t}}|\Na \theta_{s,
t}|^2, \quad t\in [0, T].\label{lem:main30:eq1} \eeq Thus, we have \beq \pd {}t\int_{L_{s, t}}\;|H_{s, t}|^2\leq
-2(\dd_0-\La\ee_0)\int_{L_{s,t}}\;|H_{s, t}|^2,\quad t\in [0,
T].\label{lem:main31:eq7} \eeq
\end{lem}

\begin{proof} Suppose that (\ref{lem:main30:eq1}) doesn't hold, there exist some constants
 $s_i\ri0, \dd_i\ri 0$ and $t_i\in [0, T]$ such that
\beq |A_{s_i}|(t)\leq \La, \quad |H_{s_i}|(t)\ri0, \quad t\in [0, T]
\label{lem:main31:eq2} \eeq and \beq \int_{L_{s_i, t_i}}\;|\Delta
\theta_{s_i, t_i}|^2\leq(\frac {\bar R}{2n}+\dd_i)\int_{L_{s_i,
t_i}}|\Na \theta_{s_i, t_i}|^2.\label{lem:main31:eq1} \eeq
 By (\ref{lem:main31:eq2}) and Proposition \ref{compactness} a sequence
of the Lagrangian mean curvature flow $L_{s_i, t}(t\in (0, T))$ will
converge to a limit Lagrangian mean curvature flow $L_{\infty, t}$
smoothly for $t\in (0, T)$. Since the initial submanifolds $L_{s_i}$
satisfies (\ref{main3:eq1}),  the limit flow has $L_{\infty}(t)\ri
L_{0}$ in $C^{2, \al}$ as $t$ goes to zero. By
(\ref{lem:main31:eq2})again the mean curvature of
$L_{\infty}(t)(t\in (0, T))$ are identically zero and  by the
uniqueness of mean curvature flow we have
$$L_{s_i, t}\ri L_{\infty, t}=L_0,\quad t\in [0, T].$$

Note that by (\ref{lem:main31:eq1}) we have \beq \int_{L_{s_i,
t_i}}\;|\Delta \frac 1{s_i}\theta_{s_i, t_i}|^2\leq(\frac {\bar
R}{2n}+\dd_i)\int_{L_{s_i, t_i}}|\Na \frac {1}{s_i}\theta_{s_i,
t_i}|^2.\n \eeq Since $L_0$ is minimal, we can take $s_i\ri0$ to get
\beq \int_{L_{0}}\;\Big|\Delta \pd {\theta_{s, t}}s\Big|_{(0,
t_i)}\Big|^2\leq\frac {\bar R}{2n}\int_{L_{0}}\;\Big|\Na \pd
{\theta_{s, t}}s\Big|_{(0, t_i)}\Big|^2.\label{lem:main31:eq3} \eeq

\medskip

On the other hand, by Lemma \ref{lem:angle} the Lagrangian angle  $\theta_{s, t}$ satisfies
\beq
\pd {\theta_{s, t}}t=\Delta_{s, t} \theta_{s, t}+\frac {\bar R}{2n}\theta_{s, t}.\n
\eeq
Since $L_0$ is minimal, we can take $s=s_i\ri 0$ to derive
\beq
\pd {}t\pd {\theta_{s, t}}s\Big|_{s=0}=\Delta_0\pd {\theta_{s, t}}s\Big|_{s=0}+\frac {\bar R}{2n}\pd {\theta_{s, t}}s\Big|_{s=0}. \label{lem:main31:eq4}
\eeq
By the eigenvalue decomposition as in
Lemma \ref{lem:positive}, we have
$$\pd {\theta_{s, t}}s\Big|_{(s, t)=(0, 0)}=-\Delta_0f_0-\frac {\bar R}{2n}f_0\perp E_{\la_1}.$$
Now we claim that \beq \pd {\theta_{s, t}}s\Big|_{s=0}\perp
E_{\la_1},\quad t\in [0, T]. \label{lem:main31:eq5} \eeq In fact,
for any function $\eta\in E_{\la_1}$ by (\ref{lem:main31:eq4}) we
have \beqs
\pd {}t\int_{L_0}\;\eta\pd {\theta_{s, t}}s\Big|_{s=0}&=&\int_{L_0}\;\eta \Big(\Delta_0\pd {\theta_{s, t}}s\Big|_{s=0}+\frac {\bar R}{2n}\pd {\theta_{s, t}}s\Big|_{s=0}\Big)\\
&=&\int_{L_0}\;\Big(\Delta_0\eta+\frac {\bar R}{2n} \eta\Big)\pd {\theta_{s, t}}s\Big|_{s=0}\\
&=&0,
\eeqs
which proves (\ref{lem:main31:eq5}).

Note that
\beq
\pd {\theta_{s, t}}s\Big|_{(s, t)=(0, 0)}=-\Delta_0f_0-\frac {\bar R}{2n}f_0\neq0,\n
\eeq
since $X$ is an essential hamiltonian variation.
Thus, by (\ref{lem:main31:eq4}) we can see that $\pd {\theta_{s, t}}s\Big|_{s=0}$ is nonzero for all $t\in [0, T]$, and (\ref{lem:main31:eq5}) implies
that
\beq
\int_{L_{0}}\;\Big|\Delta \pd {\theta_{s, t}}s\Big|_{s=0}\Big|^2\geq \la_2\int_{L_{0}}\;\Big|\Na \pd {\theta_{s, t}}s\Big|_{s=0}\Big|^2,\label{lem:main31:eq6}
\eeq
where the second eigenvalue $\la_2>\la_1=\frac {\bar R}{2n}.$ Note that (\ref{lem:main31:eq6})
 contradicts (\ref{lem:main31:eq3}), and (\ref{lem:main30:eq1})  is proved.
\medskip

Recall that by (\ref{lem:H:eq2}) in Lemma \ref{Lem:H} we have
\beqn \pd
{}t\int_{L_{s, t}}\;|H|^2&\leq&\int_{L_{s, t}}\;-2|\Na_i H^i|^2+\frac {\bar R}{n}|H|^2+2\La\ee_0|H|^2\n\\
&=&\int_{L_{s, t}}\;-2|\Delta \theta_{s, t}|^2+\Big(\frac {\bar R}{n}+2\La\ee_0\Big)|\Na\theta_{s, t}|^2\n\\
&\leq &-2(\dd_0-\La\ee_0)\int_{L_{s, t}}\;|H|^2.\label{lem:main31:eq8} \eeqn
Thus, (\ref{lem:main31:eq7}) is proved.
\end{proof}

\subsection{Proof of Theorem \ref{main3}}

In this section, we will prove Theorem \ref{main3} by using the same argument as in the proof
of Theorem \ref{main} and \ref{main2}. Since $L_0$ is a
smooth minimal Lagrangian submanifold, we can find $\kappa_0, r_0>0$
such that $L_0$ is $2\kappa_0$-noncollapsed on the scale $r_0.$ Thus, by the assumption of Theorem
\ref{main3} we can choose
$\ee_0$ small enough such that $L_s$ is $\kappa_0$-noncollapsed on the scale $r_0$ and
$L_s\in \cA(\kappa_0, r_0, \La_0, \ee_0)$ for some constant $\La_0>0$, where
$\cA(\kappa, r, \La, \ee)$ is the
following subspace of Lagrangian submanifolds in $M$ defined by
$$\cA(\kappa, r, \La, \ee)=\Big\{L \;\Big|\;L \hbox{ is $\kappa$-noncollapsed on the scale $r$ with}\;
|A|(t)\leq \La, \;|H|(t)\leq \ee\Big\}.$$

\medskip

Consider the solution
$L_{s, t}$ of the Lagrangian mean curvature flow with the initial
data $L_s$, we have
\medskip

\begin{lem}\label{lem:initial:main3}If the initial Lagrangian submanifold $L_s\in \cA(\kappa, r, \La, \ee)$,
then there exists $\tau=\tau(n, \La, K_1)$ such that $L_{s, t}\in
\cA(\frac 12\kappa, r, 2\La, 2\ee)$ for $t\in [0, \tau].$
\end{lem}
\begin{proof}This result follows directly from Lemma \ref{lem:vol} and Lemma \ref{lem:zero}.
\end{proof}

\begin{lem}\label{lem:main3}For any $\kappa_0, r_0, \La_0, V_0, T>0$ there exists $\ee_0=\ee_0(\kappa_0, r_0, \La_0, n, K_5, V_0)$
such that if the solution $L_{s, t}(t\in [0, T])$ of the Lagrangian mean
curvature flow satisfies
\begin{enumerate}
  \item  $L_s\in \cA(\kappa_0, r_0, \La_0, \ee_0)$ and $\vol(L_s)\leq V_0$,

  \item $L_{s, t}\in \cA(\frac 13\kappa_0, r_0, 6\La_0, 2\ee_0^{\frac 1{n+2}})(t\in [0, T])$,
\end{enumerate}
Then we have the following properties
\begin{enumerate}
  \item[(a)] The mean curvature vector satisfies
  $$\max_{L_{s, t}}|H_{s, t}|\leq \ee_0^{\frac 1{n+2}}e^{-\frac {\dd_0}{n+2}t},\quad t\in [\tau, T].$$
  \item[(b)] The second fundamental form $$\max_{L_{s, t}}|A_{s, t}|\leq 3\La_0,\quad t\in [0, T].$$
  \item[(c)] $L_{s, t}$ is $\frac 23\kappa_0$-noncollapsed on the scale $r_0$ for $t\in [0, T].$
\end{enumerate}
Thus, the solution $L_{s, t}\in\cA(\frac 23\kappa_0, r_0, 3\La_0,
\ee_0^{\frac 1{n+2}})$ for $t\in [0, T],$ and by Lemma
\ref{lem:initial:main3} we can extend the solution to $[0, T+\dd]$ such
that $L_{s, t}\in \cA(\frac 13\kappa_0, r_0, 6\La_0, 2\ee_0^{\frac
1{n+2}})(t\in [0, T+\dd])$ for some $\dd=\dd(n, \La_0, K_1)>0.$
\end{lem}
\begin{proof}$(a)$. For any $\La_0>0,$ by Lemma \ref{lem:main31} we can choose $\ee_0=\ee_0(X, \La, L_0, M)$ small enough such that
 the mean curvature vector satisfies
\beq\int_{L_{s, t}}\; |H_{s, t}|^2d\mu_{s, t}\leq e^{-\dd_0t}\int_{L_s}\;
 |H_s|^2d\mu_s\leq V_0\ee_0^2e^{-\dd_0 t},\quad t\in [0, T].\label{lem:main3:eq1}\eeq
Note that $L_{s, t}\in \cA(\frac 13\kappa_0, r_0, 6\La_0,
2\ee_0^{\frac 1{n+2}})$ for $t\in [0, T]$, by Lemma \ref{lem:higer}
and Remark \ref{rem} there is a constant $C_1=C_1(n, \La_0, K_2)$
such that \beq |\Na A_{s, t}|\leq C_1(n, \La_0, K_2, \tau),\quad
t\in [\tau, T].\label{lem:main3:eq2} \eeq Here we can choose
$\tau=\tau(n, \La_0, K_1)$ in Lemma \ref{lem:initial:main3}. Thus,
by Lemma \ref{lem:f} and (\ref{lem:main3:eq1})(\ref{lem:main3:eq2})
we have \beq |H_{s, t}|\leq \Big(\sqrt{\frac
3{\kappa_0}}+C_1\Big)V_0^{\frac 1{n+2}}\ee_0^{\frac 2{n+2}}e^{-\frac
{\dd_0}{n+2}t},\quad t\in [\tau, T]. \eeq where we have used the
fact that $L_t$ is $\frac {\kappa}3$-noncollapsed on the scale $r_0$
and $V_0\ee_0^2\leq r_0^{n+2}$ if $\ee_0$ is small enough. Thus, if
$\ee_0$ is  small such that $\Big(\sqrt{\frac
3{\kappa_0}}+C_1\Big)V_0^{\frac 1{n+2}}\ee_0^{\frac 1{n+2}}\leq 1,$
then we have \beq |H_{s, t}|\leq \ee_0^{\frac 1{n+2}}e^{-\frac
{\dd_0}{n+2}t},\quad t\in [\tau, T].\n \eeq

$(b)$. By Lemma \ref{lem:higer} and Remark \ref{rem} there exist
some constants $C_k=C_k(n, \La_0, K_{k+1})$ such that \beq |\Na^k
A_{s, t}|\leq C_{k}(n, \La_0, K_{k+1}, \tau),\quad t\in [\tau,
T].\label{lem:main3:eq3} \eeq By Lemma \ref{lem:higer} and Property
$(a)$, we have \beq \int_{L_{s, t}}|\Na^2 H_{s, t}|^2d\mu_{s, t}\leq
\int_{L_{s, t}}|H_{s, t}||\Na^4H_{s, t}|d\mu_{s, t}\leq
V_0C_4\ee_0^{\frac 1{n+2}} e^{-\frac {\dd_0}{n+2}t} ,\quad t\in
[\tau, T], \n\eeq where we used the fact that $\vol(L_{s, t})$ is
decreasing along the flow. Thus, by Lemma \ref{lem:f} we have \beq
|\Na^2H_{s, t}|\leq \Big(\sqrt{\frac 3{\kappa_0}}+C_3\Big)C_4^{\frac
1{n+2}}V_0^{\frac 1{n+2}} \ee_0^{\frac 1{(n+2)^2}}e^{-\frac
{\dd_0}{(n+2)^2}t},\quad t\in [\tau, T].\label{lem:main3:eq4}\eeq
Recall that by Lemma \ref{lem:LH} $|A|$ satisfies the inequality
\beq \pd {}t|A|\leq |\Na^2H|+c(n)|A|^2|H|+|\bar
Rm||H|.\label{lem:main3:eq5} \eeq Thus, by Lemma
\ref{lem:initial:main3}, (\ref{lem:main3:eq4})(\ref{lem:main3:eq5})
and $(a)$ we have \beqn
|A_{s, t}|&\leq&|A_{s, \tau}|+\int_{\tau}^t |\Na^2H_{s, t}|+(K_0+|A|^2)|H_{s, t}|\n\\
&\leq &2\La_0+\Big(\sqrt{\frac 3{\kappa_0}}+C_3\Big)C_4^{\frac
1{n+2}}V_0^{\frac 1{n+2}}
\ee_0^{\frac 1{(n+2)^2}} \frac {(n+2)^2}{\dd_0}\n\\
&&+(K_0+36\La_0^2)\ee_0^{\frac 1{n+2}}\frac {n+2}{\dd_0}\n\\
&\leq&3\La_0, \eeqn if we choose $\ee_0$ sufficiently small.

(c). By (\ref{lem:vol:eq1}), Lemma \ref{lem:initial:main3} Property $(a)(b)$
we have
\beqs E(t)&\leq& \int_{0}^{\tau}\;\max_{L}(|A||H|+|H|^2)\;ds+\int_{\tau}^t\;\max_{L}(|A||H|+|H|^2)\;ds\\
&\leq&4\La_0\ee_0\tau+4\ee_0^2\tau+3\La_0\ee_0^{\frac 1{n+2}} \frac {n+2}{\dd_0}+\ee_0^{\frac 2{n+2}}\frac {n+2}{2\dd_0}\\
&\leq &\frac 1{n+1}\log \frac 32, \quad t\in [0, T], \eeqs where
$\ee_0$ is small enough. Thus, by Lemma \ref{lem:vol} $ L_{s, t}$ is
$\frac 23\kappa_0$-noncollapsed on the scale $r_0$ for $t\in [0,
T].$

\end{proof}

Now we can finish the proof of Theorem \ref{main3}.

\begin{proof}[Proof of Theorem \ref{main3}]. Suppose that $L_s\in \cA(\kappa_0, r_0, \La_0, \ee_0)$
for any positive constants $\kappa_0, r_0, \La_0$ and small $\ee_0$
which will be chosen later. Define
$$t_0=\sup\Big\{t>0\;\Big|\; L_{s, \xi}\in \cA(\frac 13\kappa_0, r_0, 6\La_0, 2\ee_0^{\frac 1{n+2}} ),\;\;\xi\in [0, t)\Big\}.$$
Suppose that $t_0<+\infty.$ By Lemma \ref{lem:main3}, there exists
$\ee_0=\ee_0(\kappa_0, r_0, \La_0, n, K_5, V_0)>0$ such that
$L_{s, t}\in\cA(\frac 23\kappa_0, r_0, 3\La_0, \ee_0^{\frac 1{n+2}})$ for
all $t\in [0, t_0).$ Moreover, by Lemma \ref{lem:main3} again the
solution $L_t$ can be extended to $[0, t_0+\dd]$ such that $L_{s, t}\in
\cA(\frac 13\kappa_0, r_0, 6\La_0, 2\ee_0^{\frac 1{n+2}})$, which
contradicts the definition of $t_0$. Thus, $t_0=+\infty$ and
$$L_{s, t}\in \cA(\frac 13\kappa_0, r_0, 6\La_0, 2\ee_0^{\frac 1{n+2}}
),\quad t\in [0, \infty).$$ By Lemma \ref{lem:main3} the mean
curvature vector will decay exponentially to zero and the flow will
converge to a smooth minimal Lagrangian submanifold. The theorem is
proved.

\end{proof}

\section{Examples}
In this section, we give some examples of minimal Lagrangian manifolds where Theorem \ref{main} and Theorem
\ref{main3} can be applied. However, to the author's knowledge, there is no examples of strictly hamiltonian
stable minimal Lagrangian submanifold in K\"ahler-Einstein manifolds with positive scalar curvature.

\medskip
\textbf{Example 1:} (cf. \cite{[Lee1]}) Let $M_1, M_2$ be closed Riemann surfaces with hyperbolic metrics
$g_1, g_2$ respectively. Then $(M_1, g_1)\times (M_2, g_2)$ is a K\"ahler-Einstein surface of negative scalar
curvature. Suppose that $\Si$ be a closed surface with $\chi(\Si)=p_1 \chi(M_1)=p_2 \chi(M_2)$ where
$p_1, p_2$ are positive constants and the map
$$ f=(f_1, f_2): \Si\ri (M_1, g_1)\times (M_2, g_2) $$
satisfies
$\deg f_1=p_1,\; \deg f_2=-p_2$ or $\deg f_1=-p_1,\;\deg f_2=p_2$. Then there exists a unique minimal
Lagrangian surface $L_0$ in the homotopy class $f$. By Theorem \ref{main}, for any small Lagrangian
perturbation
of $L_0$ as the initial data, the mean curvature flow will exist for all time and converge exponentially to
$L_0.$

\medskip
\textbf{Example 2:} (cf. \cite{[Oh]}\cite{[Oh2]}) Consider the Clifford torus
$$\TT^{n}=\{[z_0: z_1: \cdots :z_n]\in \CC\PP^n\;|\:|z_0|=|z_1|=\cdots=|z_n|\} .$$
It is proved in \cite{[Oh]} that the Clifford torus is hamiltonian stable and the first eigenvalue of the
Laplacian is $\la_1=\frac {\bar R}{2n}.$ By \cite{[Oh2]} the first eigenspace is spanned by the
following functions restricted to the torus:
\beq  \Re (z_i), \;\Im(z_i),\;\Re(z_i\bar z_j),\;\Im(z_i\bar z_j) \label{ex:eq1}\eeq
for $0\leq i\neq j\leq n.$ Thus, if the initial data is any small hamiltonian deformation of $\TT^n$ generated by a vector
field $X=J\Na f$ where $f$ is not in the space spanned by (\ref{ex:eq1}), the mean curvature flow will
exists for all time and deform it exponentially to a Clifford torus up to congruence by Theorem \ref{main3}.

\medskip

More generally, we have the following example where Theorem \ref{main3} can be applied:
\medskip

\textbf{Example 3:} (cf. \cite{[Ono]}) Let $G$ be a compact semisimple Lie group, $\frak g$ its Lie algebra,
$( , )$ an $Ad_G$-invariant inner product on $\frak g$, and $M$ an adjoint orbit in $\frak g$ with the associate
2-form equal to the canonical symplectic form. If $(M, ( , ))$ is K\"ahler-Einstein with positive scalar curvature and
$L\subset M$ is a closed minimal Lagrangian submanifold, then $\la_1=\frac {\bar R}{2n}$ and $L$ is hamiltonian
stable. Morover, all of the coordinate functions of $L\ri \frak g$ are in the first eigenspace of $L.$ Thus, as
in Example 2, Theorem \ref{main3} can be applied in this situation.

\noindent
Department of Mathematics \\
University of Science and Technology of China, 230026, Anhui province, China. \\
Email: hzli@ustc.edu.cn\\

\end{document}